\documentclass[reqno, 11pt]{jams-l}

\usepackage{amssymb}
\usepackage{amsthm}
\usepackage[colorlinks,citecolor=red,pagebackref,hypertexnames=false]{hyperref}
\usepackage{color}
\usepackage{esint}
\usepackage{bbm}
\usepackage{graphicx}
\usepackage{datetime}
\usepackage[left=3.3cm, right=3.3cm, top=3cm, bottom=3cm]{geometry}

\def\dB{{\mathcal{B}}}

\def\dF{{\mathcal{F}}}
\def\dG{{\mathcal{G}}}
\def\dN{{\mathcal{N}}}
\def\dS{{\mathcal{S}}}

\def\bS{{\mathbb{S}}}

\DeclareMathOperator{\diam}{diam}
\def\Lip{\mathop\mathrm{Lip}} 						
\def\dist{\mathop\mathrm{dist}} 						

\newcommand{\ps}[1]{\left( #1 \right)}

\newcommand{\ck}[1]{\left\{#1 \right\}}
\newcommand{\av}[1]{\left| #1 \right|}

\def\XXint#1#2#3{{\setbox0=\hbox{$#1{#2#3}{\int}$ }
\vcenter{\hbox{$#2#3$ }}\kern-.58\wd0}}

\def\grad{\nabla}

\newtheorem{theorem}{Theorem}[section]
\newtheorem{lemma}[theorem]{Lemma}

\theoremstyle{definition}
\newtheorem{definition}[theorem]{Definition}

\theoremstyle{remark}
\newtheorem{remark}[theorem]{Remark}

\numberwithin{equation}{section}

\newcommand{\R}{\mathbb{R}}
\newcommand{\N}{\mathbb{N}}
\newcommand{\Z}{\mathbb{Z}}
\newcommand{\C}{\mathbb{C}}

\newcommand{\hd}{\mathcal{H}^d}
\newcommand{\hn}{\mathcal{H}^n}

\newcommand{\B}{\mathbb{B}}

\newcommand{\Q}{\mathbb{Q}}

\newcommand{\spt}{\mathrm{spt}}
\newcommand{\qpo}{L_Q^0}
\newcommand{\qpx}{L_Q^x}

\newcommand{\dt}{\frac{dt}{t}}

\newcommand{\wavi}{\varphi_I}
\newcommand{\wavj}{\varphi_J}
\newcommand{\ncubes}{\mathcal{D}_n}
\newcommand{\mucubes}{\mathcal{D}_\mu}
\newcommand{\dcubes}{\mathcal{D}_d}

\newcommand{\pis}{\Pi^{\circlearrowright}}
\newcommand{\pisp}{\Pi^{\circlearrowright}_{\R^n}}

\newcommand{\pixo}{\Pi^{\circlearrowright}_{L_Q^0}}
\newcommand{\pix}{\Pi^{\circlearrowright}_{L_Q^x}}
\newcommand{\pil}{\Pi^{\circlearrowright}_{L^x}}
\newcommand{\poxo}{\Pi_{L_Q^0}}
\newcommand{\cmu}{C_{\mu}}
\newcommand{\cmufi}{C_{\mu, \phi}}
\newcommand{\cmuk}{C_{\mu, k}}
\newcommand{\cmuef}{C_{f\mu}}
\newcommand{\cmuq}{C_{\mu, J(Q)}}

\newcommand{\psit}{\Phi_t}

\newcommand{\gio}{g_i^{0}}

\newcommand{\quoty}{\ps{\frac{y}{t}}}

\newcommand{\quotyx}{\ps{\frac{y-x}{t}}}
\newcommand{\quotxy}{\ps{\frac{x-y}{t}}}
\newcommand{\quotp}{\frac{\Pi^{\circlearrowright}_{L_Q^0}(y)}{t}}

\newcommand{\quotr}{\frac{\Pi^{\circlearrowright}_{\R^n}\left(R_{L_Q^x}(y-x)\right)}{t}}
\newcommand{\quotpor}{\frac{\poxo(y)}{t}}
\newcommand{\bigcubes}{\dB\dN\dG}
\newcommand{\smallcubes}{\dS\dN\dG}
\newcommand{\rot}[2]{R_{L_{#2}^{#1}}}
\newcommand{\neig}{\mathrm{Neig}}

\newcommand{\chara}{\mathbbm{1}}

\newcommand{\eqt}[1]{\ensuremath{\stackrel{#1}{=}}}
\newcommand{\leqt}[1]{\ensuremath{\stackrel{#1}{\leq}}}

\newcommand{\lesssimt}[1]{\ensuremath{\stackrel{#1}{\lesssim}}}

\newcommand{\good}{\mathrm{\textbf{Good}}}
\newcommand{\term}{\mathrm{\textbf{Term}}}

\newcommand{\Stop}{\mathop\mathrm{\textbf{Stop}}}
\newcommand{\tree}{\mathop\mathrm{\textbf{Tree}}}
\newcommand{\uptree}{\mathop\mathrm{\textbf{upTree}}}
\newcommand{\reg}{\mathop\mathrm{\textbf{Reg}}}

\newcommand\blfootnote[1]{%
  \begingroup
  \renewcommand\thefootnote{}\footnote{#1}%
  \addtocounter{footnote}{-1}%
  \endgroup
}

\numberwithin{equation}{section}
\theoremstyle{plain}
\newtheorem{sublemma}[theorem]{Sublemma}
\newtheorem{corollary}[theorem]{Corollary}
\newtheorem{notation}[theorem]{Notation}
\newtheorem{proposition}[theorem]{Proposition}

\title[Center of mass]{A square function involving the center of mass \\and rectifiability}

\author{Michele Villa}
\address{Research Unit of Mathematical Sciences, University of Oulu. P.O. Box 8000, FI-90014, University of Oulu, Finland.}
\email{michele.villa "at" oulu.fi}

\dedicatory{}
\begin{document}

\begin{center}
\maketitle

\begin{minipage}[c][][r]{300pt}
\begin{small}
\textsc{Abstract.} For a Radon measure $\mu$ on $\R^d$, define 
$C^n_\mu(x, t)= \left(\frac{1}{t^n} \left|\int_{B(x,t)} \frac{x-y}{t} \, d\mu(y)\right| \right)$. This coefficient quantifies how symmetric the measure $\mu$ is by comparing the center of mass at a given scale and location to the actual center of the ball. We show that if $\mu$ is $n$-rectifiable, then 
\begin{align*}
    \int_0^\infty |C^n_\mu(x,t)|^2 \dt < \infty \enskip \mu\mbox{-almost everywhere}.
\end{align*}
Together with a previous result of Mayboroda and Volberg, where they showed that the converse holds true, this gives a new characterisation of $n$-rectifiability. To prove our main result, we also show that for an $n$-uniformly rectifiable measure, $|C_\mu^n(x,t)|^2 \dt d\mu$ is a Carleson measure on $\spt(\mu) \times (0,\infty)$. We also show that, whenever a measure $\mu$ is $1$-rectifiable in the plane, then the same Dini condition as above holds for more general kernels. We also give a characterisation of uniform 1-rectifiability in the plane in terms of a Carleson measure condition. This uses a classification of $\Omega$-symmetric measures from \cite{villa19}. 
\end{small}
\end{minipage}
\end{center}

\blfootnote{\today \\ \textup{2010} \textit{Mathematics Subject Classification}: \textup{28A75}, \textup{28A12} \textup{28A78}.

 \noindent
\textit{Key words and phrases.} Rectifiability, tangent points, beta numbers, Hausdorff content.
\noindent
M. Villa was supported by The Maxwell Institute Graduate School in Analysis and its
Applications, a Centre for Doctoral Training funded by the UK Engineering and Physical
Sciences Research Council (grant EP/L016508/01), the Scottish Funding Council, Heriot-Watt
University and the University of Edinburgh.
}
\tableofcontents
\section{Introduction}
Rectifiable sets are a main object of study in geometric measure theory. These are sets that can be covered by (countably many) Lipschitz images of the Euclidean space, up to a set of Hausdorff measure zero. 
More generally, a \textit{measure} $\mu$ in $\R^d$ is $n$-rectifiable if there exists an $n$-rectifiable set $E$ and a Borel function $f: \R^d \to \R$ so that $\mu = f \, \hn|_E$.

In the first part of this paper, however, we will use a quantitative notion of rectifiability: that of \textit{uniformly rectifiable} (or UR) sets. First introduced by Guy David and Stephen Semmes in \cite{singularintegrals}, uniform rectifiability is a stronger property than rectifiability: one has control on how much mass in any given ball centered on the set can be covered by \textit{just one} Lipschitz image - in contrast to rectifiable sets, where in some places one may need very many images to cover just a small portion of the original set. This sub-area of GMT has very strong ties to Harmonic analysis. Indeed, not only `\textit{is there an obvious analogy between rectifiability property of sets} (this is GMT) \textit{and differentiability properties of functions} (see \cite{analysisofandon}, Introduction), but also the problems themselves that originally motivated the development of the theory of UR sets are harmonic analytic in nature, although this problems concerns more singular integrals and analytic capacity. Here's an example: consider $f: \R \to \R$ in $L^2(\R)$. Then (see \cite{stein2}) $f$ is locally absolute continuous and $f' \in L^2$ if and only if $\int_0^\infty \int_\R t^{-2} |f(x+t)+ f(x-t) - 2f(x)|^2 \, dx \frac{dt}{t} < \infty$. The quantity  $t^{-1}(f(x+t)+ f(x-t) - 2f(x))$ can be seen as a measurement of how far  $f$ is from being affine. Something along these lines is the following result by Dorronsoro (\cite{do1}). 
	Let $f \in W^{1,2}(\R^d)$ and define $
		\Omega_f (x, t) := \left(\inf_A  t^{-d}\int_{B(x, t)} t^{-2}\left(f(y) - A(y)\right)^2
		\, dy \right)^{\frac{1}{2}},
	$
	where the infimum is taken over all affine functions $A$.
	Then $ \int_{\R^d} \int_0^\infty \Omega_f(x, t)^2 \, \frac{dt}{t} dx \sim \| \nabla f
		\|_2^2.$
Since rectifiable sets are composed of Lipschitz images, it is natural to ask whether similar quantities can be designed for sets. Take a subset $E \subset \R^n$ and let
\begin{align} \label{eq:01_beta}
\beta_{E, p}^n(x,t):=\inf_L \left(\frac{1}{t^n} \int_{B \cap E} \left(\frac{\dist(y,
L)}{t}\right)^p  d \hn(y) \right)^{\frac{1}{p}},
\end{align}
A similar quantity was first introduced by Peter Jones in \cite{jones} as a tool to prove his Analyst's Traveling Salesman theorem (and so it is called the Jones-beta number), and further developed by David and Semmes in \cite{singularintegrals}, \cite{analysisofandon}, to show that if a set $E$ is $n$-Ahlfors - David regular (ADR), that is, if there exists a constant $c\geq 1$ such that $c^{-1}\, r^n \hn(B(x, r) \cap E) \leq c \, r^n$ for all $x \in E$, $r>0$, then $E$ is UR if and only if 
\begin{align} \label{eq:01_betacarleson}
    \int_{B_R} \int_0^R \beta_{E, p}^n (x, t)^2 \, \dt \, d \hn|_E(x) \leq C\, R^n.
\end{align}
Note the analogy with the result by Dorronsoro: there we studied regularity properties of functions, here those of sets - in both cases such properties are characterised quantitatively by measuring the linear approximation properties of our objects.
David and Semmes also show that a set $E$ is UR if and only if a wide class of Calder\'{o}n-Zygmund kernels are $L^2(E)$ bounded; this is precisely the problem which originated the field: on what sort of sets are Calder\'{o}n - Zygmund kernels $L^2$ bounded?

Let us go back for a moment to the world of functions: we may find further inspiration. Recently, R. Alabern, J. Mateu and J. Verdera showed in \cite{mateu-verdera}  that if $f \in W^{1, p}(\R^d)$, then $\|S(f)\|_p \sim \|\nabla f\|_p$, where
$$S(f)^2(x):= \int_0^\infty \left| t^{-d} \int_{B(x, t)} t^{-1}(f(x)- f(y)) \, dy \, \right|^2 \, \dt.$$
Appealing to analogy once more (\textit{`you can prove anything using the Analogy'}\footnote{See U.K. Le Guin, \textit{The Dispossessed}. However, the reader should be warned: analogy can be misleading, see \cite{singularintegrals}, Introduction.}), one may define a corresponding quantity for sets, or, more generally for a Radon measure $\mu$ on $\R^d$: we let the \textit{$n$-dimensional $C$-number} to be defined as
\begin{align*}
    C^n_\mu(x, t)^2= C_\mu(x,t)^2 := \left(\frac{1}{t^n} \left|\int \frac{x-y}{t} \, d\mu(y)\right| \right)^2.
\end{align*}
While a $\beta$ number gives us information on the linear approximation properties of a set or a measure, $C_\mu$ tells us about the geometry of measures by capturing how far the center of mass of our object is from the center of the ball where we are focusing our attention.
This quantity appears for different reasons other than pure analogy: Mattila, in \cite{mattila_cauchy}, while investigating for what kind of measures $\mu$ in $\C$ does the Cauchy transform exists $\mu$-almost everywhere (in the sense of principal values), gives a complete characterisation of what he calls \textit{symmetric} measures; that is, measures that satisfy $
    C_\mu(x, t) = 0 \, \mbox{ for all } x \in \spt(\mu) \mbox{ and for all } t>0.
$
He shows that any symmetric locally finite Borel measure on $\C$ is either discrete or coninuous. In the latter case, it is either the $2$-dimensional Lebesgue measure (up to a multiplicative constant) or a countable sum of $1$-dimensional Hausdorff measures restricted to equidistant affine lines. Mattila needed such characterisation to understand the geometry of tangent measures of a measure $\mu$ for which the Cauchy transform exists $\mu$-almost everywhere (in the sense of principal values) - thus to understand the geometry of $\mu$ itself. 
Briefly after, Mattila and Preiss (see \cite{Mattila-Preiss}) generalised this to the higher dimensional equivalent. 

More recently, Mayboroda and Volberg in \cite{mayvol} proved the following.
\begin{theorem} \label{theorem:mayvol}
Let $\mu$ be a finite measure with  finite and positive $n$-dimensional upper density $\mu$-almost everywhere. If 
\begin{align}
    \int_0^\infty |C^n_\mu(x,t)|^2 \dt <\infty\,\,  \mbox{ for } \mu-\mbox{almost all } x \in \R^d, 
\end{align}
then $\mu$ is $n$-rectifiable.
\end{theorem}
Recall that the $n$-dimensional upper density of $\mu$ at $x$ is given by 
\begin{align}
    \theta^{*,n}(x,\mu):= \limsup_{r \to 0} \frac{\mu(B(x, r))}{r^n}.
\end{align}
The same result appeared as a corollary of the work of Jaye, Nazarov and Tolsa in \cite{jnt}, Subsection 1.6.

In this paper we prove the converse.
\begin{theorem} \label{theorem:main_rectifiability}
Let $\mu$ be an $n$-rectifiable measure in $\R^d$. Then 
\begin{align}
    \int_0^\infty |C^n_\mu(x,t)|^2 \dt <\infty \, \, \mbox{ for } \mu-\mbox{almost all } x \in \R^d, 
\end{align}
\end{theorem}
Thus, together with Theorem  \ref{theorem:mayvol}, we have a characterisation of rectifiability.
\begin{corollary}
Let $\mu$ be a measure on $\R^d$ so that  $0< \theta^{*,n}(x,\mu)< \infty$ for $\mu$-almost all $x \in \R^d$. Then $\mu$ is $n$-rectifiable if and only if 
\begin{align}
    \int_0^\infty |C^n_\mu(x,t)|^2 \dt <\infty \enskip \enskip \mu-\mbox{almost all } x \in \R^d, 
\end{align}
\end{corollary}
Theorem \ref{theorem:main_rectifiability} will follow from the other result of this paper.
\begin{theorem} \label{theorem:main_UR}
Let $\mu$ be an $n$-ADR measure on $\R^d$. If $\mu$ is uniformly rectifiable then, for each ball $B$ centered on $ \spt(\mu) $ and with radius $r_B$, 
\begin{align}
    \label{eq:01_cCarleson}
    \int_{B}\int_0^{r_B} |C_\mu(x, t)|^2 \, \dt\, d\mu(x) \lesssim r_B^n.
\end{align}
\end{theorem}
The condition in \eqref{eq:01_cCarleson} is analogous to \eqref{eq:01_betacarleson}; it says that $|C_\mu(x, t)|^2 \, d\mu(x) \dt$ is a Carleson measure on $\spt(\mu) \times (0, \infty)$. 

It is natural to ask whether a similar characterisation can be proven for more general kernels. Note that $K(x) = |x| Id\ps{\frac{x}{|x|}}$, where $Id : \bS^{d-1} \to \bS^{d-1}$ is the identity map of the sphere to itself. Now suppose we perturb the identity, in the following sense. Let $\Omega: \bS^{d-1} \to \bS^{d-1}$ be an odd, twice continuously differentiable map which is also bi-Lipshitz with constant $1+\delta_\Omega$. Given a Radon measure $\mu$,  one can define the corresponding perturbed $C_\mu^n$ number, i.e.
\begin{align*}
    C_{\Omega, \mu}^n (x,t)^2 := \ps{ \frac{1}{t^n} \left| \int \frac{|x-y| \Omega\ps{\frac{x-y}{|x-y|}}}{t} \, d\mu(y) \right| }^2. 
\end{align*}

We prove the following characterisation of uniform rectifiability in the plane. 

\begin{theorem} \label{t:UR-Omega}
Let $\Omega$ as above and suppose that $\delta_\Omega$ is sufficiently small\footnote{We require $\delta_\Omega$ to be smaller than a universal constant. For example, $\delta_\Omega < 1/10$ will work.}. Let $\mu$ be an Ahlfors $1$-regular measure on $\C$. Then $\mu$ is uniformly $1$-rectifiable if and only if $|C_{\Omega, \mu}^1(x,t)|^2 \, \dt d\mu(x)$ is a Carleson measure on $\spt(\mu) \times(0,\infty)$. 
\end{theorem}

\begin{remark}
In the case $\Omega= Id$, that the Carleson measure condition implies uniform rectifiability follows from \cite{jnt}. However, for general $\Omega$, this is novel and is an application of Theorem 1.4 in \cite{villa19}. 
\end{remark}

A counterpart of Theorem \ref{theorem:main_rectifiability} holds in this case, too. 
\begin{theorem} \label{t:R-Omega}
Let $\mu$ be a $1$-rectifiable measure on $\C$. Then 
\begin{align*}
    \int_{0}^\infty |C_{\Omega, \mu}(x,t)|^2 \, \dt < \infty \mbox{ for } \mu\mbox{-almost all } x \in \C.
\end{align*}
\end{theorem}

\begin{remark}
We will not prove Theorem \ref{t:UR-Omega} and \ref{t:R-Omega} in great details, since the proofs are very similar to those for the case where $\Omega=Id$; we will highlight the places where a slight change is needed. However, we will briefly illustrate on how to apply Theorem 1.4 from \cite{villa19} in proving one direction of Theorem \ref{t:UR-Omega}.
\end{remark}

\begin{remark}
Now, a remark added two year later, in 2022. \textit{A posteriori}, this work finds the following further motivation. In an upcoming article with J. Azzam and M. Mourgoglou \cite{a-m-v}, we show that a form of Dorronsoro's theorem holds for functions defined on uniformly rectifiable sets. As explained there, the current work could (work in progress) be used to show a variation of \cite{a-m-v}, based on \cite{mateu-verdera}, where the coefficients involved \textit{are purely metrical}. This might be of interest in the context of analysis in metric space.
\end{remark}

\subsection{Outline of the proof}
We first show that if a set is uniformly rectifiable, then we have the Carleson estimate \eqref{eq:01_cCarleson}. To show this, we follow the strategy in \cite{tolsa2014}, where the same is shown for a different square function, involving differences in densities. To prove Theorem \ref{theorem:main_rectifiability}, we then borrow the techniques and ideas from \cite{tolsatoro}, where the same is shown, but again, for the square function mentioned before. 

\subsection{Acknowledgement}
I would like to thank Jonas Azzam, my supervisor for his help and support. I would also like to thank Xavier Tolsa and Raanan Schul for useful conversations.

\section{Preliminaries}

We collect some notions and theorem from the literature and some lemmas which will be needed later. 

\subsection{Notation}
We gather here some notation and some results which will be used later on.
We write $a \lesssim b$ if there exists a constant $C$ such that $a \leq Cb$. By $a \sim b$ we mean $a \lesssim b \lesssim a$.
In general, we will use $d\in \N$ to denote the dimension of the ambient space $\R^d$, while we will use $n$, with $n\leq d-1$, to denote the `dimension' of a measure $\mu$, in the sense of $n$-Ahlfors regularity.

For sets $A,B \subset \R^n$, we let
\begin{align*}
    \dist(A,B) := \inf_{a\in A, b \in B} |a-b|.
\end{align*}
For a point $x \in \R^n$ and a subset $A \subset \R^n$, 
\begin{align*}
    \dist(x, A):= \dist(\{x\}, A)= \inf_{a\in A} \dist(x,a).
\end{align*}
We write 
\begin{align*}
    B(x, t) := \{y \in \R^n \, |\,|x-y|<t\},
\end{align*}
and, for $\lambda >0$,
\begin{align*}
    \lambda B(x,t):= B(x, \lambda t).
\end{align*}
At times, we may write $\B$ to denote $B(0,1)$. When necessary we write $B^n(x,t)$ to distinguish a ball in $\R^n$ from one in $\R^d$, which we may denote by $B^d(x, t)$. 

We will also write
\begin{align*}
\beta_\mu^{p,d}(x,t):= \beta_\mu^{p, d}(B(x,t)).
\end{align*}

Let $A \in \R^n$ and $0< \delta \leq \infty$. Set
\begin{align*}
    \mathcal{H}^d_\delta (A) := \inf \left\{\sum (\diam (A_i))^d \, |\, A \subset \cup_i A_i \, \mbox{and } \diam(A_i) \leq \delta \right\}.
\end{align*}
The $d$-dimensional Hausdorff measure of $A$ is then defined by
\begin{align*}
\hd(A) := \lim_{\delta \to 0} \mathcal{H}^d_\delta(A).
\end{align*}

\subsection{Intrinsic cubes with small boundaries}
The following construction, due to David in \cite{wavelets}, provides us with a dyadic decomposition of the support of an AD-regular measure. Such construction has been extended by Christ in \cite{christ} to spaces of homogeneous type and further refined by Hyt\"{o}nen and Martikainen in \cite{hytonen}. Here is the construction. 

\begin{theorem} \label{theorem:ADcubes}
Let $\mu$ be an $n$-AD regular measure in $\R^d$. There exists a collection $\mucubes$ of subsets $Q \subset \spt(\mu)$ with the following properties. 
\begin{enumerate}
    \item We have
    \begin{align*}
        \mucubes = \bigcup_{j \in \Z} \mucubes^j,
    \end{align*}
    where $\mucubes^j$ can be thought as the collection of cubes of sidelength $2^{-j}$. 
    \item For each $j \in \Z$, 
    \begin{align*}
        \spt(\mu) = \bigcup_{Q \in \mucubes^j} Q.
    \end{align*}
    \item If $j \leq i$, $Q \in \mucubes^j$, $Q' \in \mucubes^i$, then either $Q \subset Q'$ or else $Q \cap Q' = \emptyset$.
    \item If $j \in \Z$ and $Q \in \mucubes^j$, then there exists a constant $C_0\geq 1$ so that
    \begin{align*}
       &  C_0^{-1}2^{-j} \leq \diam(Q) \leq C_0 2^{-j}, \mbox{ and } \\
        & C_0^{-1}2^{-jn} \leq \mu(Q) \leq C_0 2^{-jn}.
    \end{align*}
    \item If $j \in \Z$, $Q \in \mucubes^j$ and $0<\tau<1$, then 
    \begin{align*}
        \mu \left(\left\{ x \in Q \, |\, \dist(x, \spt(\mu) \setminus Q) \right\}\right) \leq C \tau^{\frac{1}{C}} 2^{-nj}.
    \end{align*}
\end{enumerate}
\end{theorem}
For a proof of this, see Appendix 1 in \cite{wavelets}.

\begin{notation}
For $Q \in \mucubes^j$, we set
\begin{align}
    \ell(Q) := 2^{-j}. \label{eq:cubelength}
\end{align}
We will denote the center of $Q$ by $z_Q$. Furthermore, we set
\begin{align}
    B_Q := B(z_Q, 3 \diam(Q)).
\end{align}
\end{notation}

\subsection{The weak constant density condition}
We follow the definition given in \cite{tolsa2014}. Let $\mu$ be an $n$-AD regular measure on $\R^d$. We denote by $A(c_0, \epsilon)$ the set of points $(x, t) \subset \spt(\mu) \times (0, \infty)$ such that there exists a Borel measure $\sigma=\sigma_{x, t}$ which satisfies the following three conditions. 
\begin{enumerate}
    \item $\spt(\sigma) = \spt(\mu)$. 
    \item The measure $\sigma$ is $n$-AD regular with constant $c_0$.
    \item It holds 
    \begin{align}
       & |\sigma(B(y,r)) - r^n | \leq \epsilon \, r^n \nonumber \\
       &  \mbox{ for all } y \in \spt(\mu) \cap B(x, t) \, \mbox{ and for all } 0 < t < r.
    \end{align}
\end{enumerate}
\begin{definition} \label{definition:wcd}
A Borel measure $\mu$ on $\R^d$ is said to satisfy the weak constant density condition (WCD), if there exists a constant $c_0 >0$ such that the complement in $\spt(\mu) \times (0, \infty)$ of the set $A(c_0, \epsilon)$ defined above is a Carleson set for every $\epsilon>0$, that is, for every $\epsilon >0$, there exists a constant $C(\epsilon)>0$ so that
\begin{align*}
    \int_0^R \int_{B(x, R)} \chara_{\left(\spt(\mu) \times (0\, \infty)\right)\setminus A(c_0, \epsilon)}(x, t) \, d\mu(x) \, \dt  \leq C(\epsilon) \, R^n.
\end{align*}
for all $x \in \spt(\mu) $ and $R>0$.
\end{definition}
The WCD condition was firstly introduced by David and Semmes in \cite{singularintegrals}, Section 6. There it was proven that if a set $E$ is $n$-uniformly rectifiable, then the $n$-dimensional Hausdorff measure restricted to $E$ satisfies WCD. In that case, $\sigma_{x, t}$ was simply the push forward measure of $\hd|_E$ onto the best approximating plane. Shortly after, in \cite{analysisofandon}, they proved the converse for the dimensions $n=1,2, d-1$. More recently, Tolsa in \cite{tolsaUNI}, proved the converse for all dimension. We thus have the following theorem
\begin{theorem}[{\cite{singularintegrals}, \cite{analysisofandon}, \cite{tolsaUNI}}] \label{theorem:wcd_iff_ur}
Let $n \in (0, d)$ be an integer. An AD-regular measure $\mu$ on $\R^d$ us $n$-uniformly rectifiable if and only if it satisfies WCD.
\end{theorem}

\subsection{The $\beta$ and the $\alpha$ numbers}
\label{subsection:betaalpha}
As mentioned in the introduction, $\beta$ numbers were firstly used by Jones in \cite{jones} to understand the geometry a set or of a measure. Below we will need another quantity, introduced by Tolsa some ten years ago in \cite{tolsa2008}. They are the so called $\alpha$ numbers or coefficients, and they are defined as follows. For two Radon measures $\mu$, $\nu$ on $\R^d$, and a ball $B$ with radius $r_B$, set
\begin{align}
    \dist_B (\mu, \nu) := \sup\left\{ \left| \int f \, d\mu - \int f\, d\nu\right| \, \, | \, \Lip(f) \leq 1 \mbox{ and } \spt(f) \subset B\right\}. 
\end{align}
One can see that $\dist_B$ defines a metric on the set of Radon measures supported on $B$. Fix now an $n$-ADR measure on $\R^d$. Let $Q$ be either an instrinsic cube, i.e. $Q \in \mucubes$ or let $Q$ be an actual dyadic cube of $\R^d$ which intersects $\spt(\mu)$. Recall that $B_Q= B(z_Q, 3 \diam(Q))$ where $z_Q$ is the center of the cube. Then define
\begin{align}
    \alpha_\mu^n(Q):= \frac{1}{\ell(Q)^{n+1}} \inf_{c \geq 0, L} \dist_{B_Q}(\mu, c \hn|_L).
\end{align}
Tolsa showed the following. 
\begin{theorem}[{\cite{tolsa2008}, Theorem 1.2}] \label{theorem:tolsaalpha}
Let $\mu$ be an $n$-ADR measure. The following are equivalent:
\begin{enumerate}
    \item $\mu$ is UR.
    \item For any cube $R \in \mucubes$, we have
    \begin{align*}
        \sum_{Q \in \mucubes, Q \subset R} \alpha(Q)^2 \mu (Q) \leq C \mu(Q), 
    \end{align*}
    with $C$ independent of $R$. 
\end{enumerate}
\end{theorem}
Tolsa's motivation to introduce the $\alpha$ coefficient was, again, the study of the relationship between UR measures and $L^2(\mu)$ boundedness of Calder\'{o}n-Zygmund operators; for applications, see for example Theorem 1.3 in \cite{tolsa2014}, \cite{mastolsa2}, \cite{tolsaCauchy}.

The two coefficient $\beta$ and $\alpha$ are related by the following inequality. See \cite{tolsa2008}, Remark 3.3.
\begin{align} \label{eq:02_50}
    \frac{1}{\ell(Q)^n} \int_{B(z_Q, 2 \ell(Q)} \frac{\dist(y, L_Q)}{\ell(Q)} \, d\mu(y) \lesssim \alpha_\mu(Q).
\end{align}
We will need the following auxiliary lemma (where we merge two lemmas from \cite{tolsa2008}).
\begin{lemma}[{\cite{tolsa2008}, Lemma 5.2 and Lemma 5.4}] \label{lemma:tolsa2008}.
Let $\mu$ be a $n$-uniformly rectifiable measure on $\R^d$. For every $R \in \mucubes$, we have
\begin{align*}
    \sum_{\substack{Q \in \mucubes\\ Q \subset R}} \int_Q \left(\frac{\dist(x, L_Q)}{\ell(Q)}\right)^2 \leq C(n,C_0) \ell(R)^n.
\end{align*}
\end{lemma}

\subsection{Wavelets} \label{subsection:wavelets}
We consider a family of tensor products of Debauchies-type compactly supported wavelets with three vanishing moments (in particular, they have zero mean). They have the following properties. 
\begin{enumerate}
\item Each element of the family belongs to $C^1(\R^n)$ and it is supported on $5I$, where $I \in \ncubes$. Recall that by $\ncubes$ we denote the standard dyadic grid in $\R^n$. Hence we index the family as $\{\wavi\}_{I \in \ncubes}$.
\item For each $I \in \ncubes$, 
\begin{align} \label{eq:02_waviL2}
    \|\wavi\|_{L^2(\R^n)} = 1.
\end{align}
\item For each $I \in \ncubes$ we have
\begin{align}
    & \|\wavi\|_\infty \leq C \, \frac{1}{\ell(I)^{\frac{n}{2}}}; \label{eq:02_waviInf} \\
    & \|\nabla \wavi\|_\infty \leq C\, \frac{1}{\ell(I)^{1+ \frac{n}{2}}}. \label{eq:02_waviDifInf}.
\end{align}
\end{enumerate}
Then any function $f \in L^2(\R^n)$ can be expressed as
\begin{align*}
    f = \sum_{I \in \ncubes} \langle f, \wavi \rangle_{L^2} \, \wavi.
\end{align*}

\subsection{Preliminaries on $C_\mu$ and $\cmufi$}

We introduce a `smooth' version of the $\cmu$ numbers. The reason for doing so is that such a quantity is, in general, easier to work with. Moreover, the smooth version is, in some sense, smaller than the original version (this is because smooth cut offs can be bounded above by convex combinations of cut offs). 

For each $N \in \N$, let $\phi_N: \R^d \to \R$ be given by 
\begin{align}
    \phi_N(x) = e^{-\left| x\right|^{2N}}. \label{eq:02_30}
\end{align}
We will omit the subscript $N$ as it is unimportant for the discussion to follow.
For $t>0$, set 
\begin{align*}
\phi_t (x):= \frac{1}{t^n}\phi\left( \frac{x}{t} \right).
\end{align*}
\begin{definition}[Smooth $C$ number] \label{def:smoothed_c}
Let $\mu$ be an AD-regular measure on $\R^d$. Set 
\begin{align*}
\psit(x) := \frac{x}{t} \phi_t(x).
\end{align*}
Then we define the smooth version of $C_\mu$ as 
\begin{align*}
 \cmufi(x, t) & := \psit * \mu(x) \\
 & = \int \psit (x-y)\, d\mu(y)
\end{align*}
\end{definition}

The following it's all well known.
\begin{lemma} \label{lemma:gaussian}
Let $\mu$ be an $n$-AD-regular measure with constant $c_0$ on $\R^d$. Take $\phi_t$ as given above. 
Then 
\begin{align}
    t^n\,|\cmufi(x, t)| \leq C(n, c_0).
\end{align}
\end{lemma}
\begin{proof}
It is easy to see that
\begin{align*}
    t^n\, \int \phi_t(x-y) \, d\mu(y) \leq c_0 \Gamma\left(\frac{n+2}{2}\right) \, t^n.
\end{align*}
Let us prove this for the sake of completeness. We use the so called layer-cake decomposition:
\begin{align*}
    \int t^n\, \phi_t(x-y) \, d\mu(y) & = \int_0^1 \int_{\R^d} \chara_{\{x \in \R^d| \phi_t(x-y)>s\}}(y) \, d\mu(y) \, ds.
\end{align*}
Notice that
\begin{align*}
    \phi_t(x-y)>t \, \iff |x-y| < (-\ln(s))^{1/2} t.
\end{align*}
Thus 
\begin{align*}
\int_{\R^d} \chara_{\{x \in \R^d| \phi_t(x-y)>s\}}(y) \, d\mu(y) & = \int_{\R^d} \chara_{B(x, (-\ln(s))^{1/2} t)}(y) \, d\mu(y) \\
& \leq c_0 (-\ln(s))^{1/2} t)^n.
\end{align*}
Now, 
\begin{align*}
    \int_0^1 (-\ln(s))^{\frac{n}{s}} \, ds = \Gamma\left(\frac{n+2}{2}\right).
\end{align*}

If we now argue as in \cite{DeLellis}, we obtain the result. 
\end{proof}

\begin{remark} \label{rem:C-Omega-phi}
Clearly, the very same holds for the quantity $C_{\Omega, \mu,\phi}^n$. 
\end{remark}
 \section{Uniform rectifiability implies a Carleson condition}

This section will be devoted to prove the following proposition. 

\begin{proposition} \label{proposition:ns_ur_c}
    Let $\mu$ be an AD $n$-regular measure on $\R^d$ with constant $c_0$. If $\mu$ is uniformly $n$-rectifiable, then for each $R \in \mucubes$, we have that
    \begin{align} \label{eq:ns_ur_c1}
        \sum_{\substack{Q \in \mucubes\\Q \subset R}} \int_Q \int_{\ell(Q)}^{2\ell(Q)} |C_\mu(x,t)|^2 \, \dt d\mu(x) \lesssim_{c_0} \mu(R).
    \end{align}
\end{proposition}
The same techniques apply to prove one direction of Theorem \ref{t:UR-Omega}.
\begin{proposition}\label{p:UR-C-Omega}
    Let $\mu$ be an AD $1$-regular measure on $\C$ with constant $c_0$. If $\mu$ is uniformly $1$-rectifiable, then for each $R \in \mucubes$, we have that
    \begin{align} 
        \sum_{\substack{Q \in \mucubes\\Q \subset R}} \int_Q \int_{\ell(Q)}^{2\ell(Q)} |C^1_{\Omega,\mu}(x,t)|^2 \, \dt d\mu(x) \lesssim_{c_0} \mu(R).
    \end{align}
\end{proposition}

Let $\delta > 0$ be a small constant. To prove \eqref{eq:ns_ur_c1}, we may sum only on those cubes for which
\begin{align} \label{eq:ns_ur_cdelta}
    \alpha_\mu(1000Q) \leq \delta^2
\end{align}
holds.  
Indeed, we see that
\begin{align*}
    &\sum_{\substack{Q \in \mucubes(R) \\ \alpha_\mu(1000Q) > \delta^2}} \int_Q \int_{\ell(Q)}^{2\ell(Q)} |C_\mu(x,t)|^2\, \dt, d\mu(x) \\
    & \lesssim_{c_0} \sum_{\substack{Q \in \mucubes(R) \\ \alpha_\mu(1000Q) > \delta^2}} \int_R \int_{\ell(Q)}^{2\ell(Q)} \left(\frac{\alpha_\mu(1000Q)}{\delta^2}\right)^2\, \dt d\mu(x) \\
    & \lesssim_{c_0}  \frac{1}{\delta^4}\sum_{\substack{Q \in \mucubes(R) \\ \alpha_\mu(1000Q) > \delta^2}}  \alpha_\mu(1000Q)^2\, \mu(Q) \\
    & \leq \frac{1}{\delta^4}\sum_{\substack{Q \in \mucubes(R)}}  \alpha_\mu(1000Q)^2\, \mu(Q) \\
    & \lesssim_{\delta} \mu(R).
\end{align*}
The last inequality follows from Tolsa's Theorem 1.2 in \cite{tolsa2008}.

\begin{remark}
A recurring issue when working with the $C_\mu$ numbers is that there is no quasi-monotonicity formula (as there is for the $\beta$ numbers) as we increase the radius of the ball over which we are integrating. In other words, we do not know how $C_\mu(x, t)$ compares with $C_\mu(x, 2t)$, for example. For this reason, when wanting to compare $\mu$ with, say, the Hausdorff measure on the best approximating plane, we cannot simply consider the push forward measure through an orthogonal projection. The $C_\mu$ numbers are unstable under such space deformations. Because of this, we will use instead a \textit{circular} projection, which we learnt from \cite{mastolsa2} (see eq. 54) and from \cite{tolsa2014}.
\end{remark}

\begin{notation}
For $x=(x_1,...,x_n,x_{n+1},...,x_d) \in \R^d$, set 
\begin{align*}
    x^H := (x_1,...,x_n) \in \R^n.
\end{align*}
We define the \textit{circular projection} $\pis: \R^d \to \R^n$ by setting
\begin{align*}
    \pis(x) := \frac{|x|}{x^H} x^{H}.
\end{align*}
If we view $\R^n=V$ as a subspace of $\R^d$, we may write
\begin{align*}
\pis(x) = \frac{|x|}{|\Pi_V(x)|} \Pi_V(x), 
\end{align*}
where here $\Pi_V$ is the standard orthogonal projection onto $V$.
\end{notation}

For us, the fundamental property of $\pis$ (and the reason to use it, see the remark above) is that
\begin{align*}
|\pis(x)|= |x|.
\end{align*}
This implies, in particular, that 
\begin{align} \label{eq:06_30}
    \chara_{\B_d} = \left( \chara_{\B_n}\circ \pis\right).
\end{align}

Let $x \in \R^d$. Let $L^x$ be an affine $n$-plane such that $x \in L$. We define $\pil$ as follows. Let $R_{L^x}$ be the rotation so that $R_{L^x} (L^x)$ is parallel to $\R^n$. For $y \in \R^d$, we set
\begin{align*}
    \pil(y):= R_{L^x}^{-1} \left( \pis\left( R_{L^x} (x-y)\right) + R_{L^x}(x) \right).
\end{align*}
Again the fundamental property is satisfied.
\begin{align*}
    \left| \pil(y) - x\right|& = \left| R^{-1}_{L^x} \left( \pis\left( R_{L^x}(x-y)\right)\right) \right| 
    = |\pis(R_{L^x}(x-y))|\\ 
    & =|R_{L^x}(x-y)|=|x-y|.
\end{align*}

\bigskip

Fix a cube $Q \in \mucubes(R)$ with $\alpha_\mu(1000Q) \leq \delta^2$ and $(x, t) \in Q \times [\ell(Q), 2\ell(Q)]$. 
Let $L_Q$ be the minimising $n$-plane for $\alpha_\mu(Q)$ and $c_Q$ the minimising constant. Denote by $L_Q^x$ the $n$-plane parallel to $L_Q$ so that
\begin{align*}
 x \in L_Q^x.
\end{align*}

Let $\Omega: \bS^{d-1} \to \bS^{d-1}$ be an odd, $\mathcal{C}^2$ map which is also bi-Lipschitz. For $x \in \R^d$, set 
\begin{align*}
    K_t(x) := \frac{|x| \Omega\ps{\frac{x}{|x|}}}{t}.
\end{align*}
Note that 
\begin{align} \label{e:grad-est}
|\grad K_t(x)| \leq \frac{C(\Omega)}{t}.    
\end{align}

\begin{remark}
We carry out some estimates at this level of generality\footnote{Most of them actually hold more generally, for example it would suffice to assume that $K$ is an odd kernel with $\|\grad K\|_\infty \leq C$.}. Soon, however, we will have to restrict to the case where $\Omega= Id$ or $d=2$. We will highlight when and why this happen below.
\end{remark}
 
We write
\begin{align}\label{e:I+II}
    t^{-n} \av{\int_{B(x,t)} K_t(x-y) \, d\mu(y)} & \leq t^{-n} \av{ \int_{B(x,t)} K_t(x-y) - K_t(x- \pix(y)) \, d\mu(y) }\nonumber\\ 
    &\enskip \enskip + t^{-n}  \av{ \int_{B(x,t)} K_t(x-\pix(y)) \, d\mu(y) }\nonumber\\
    & =: I(x,t)+ II(x,t).
\end{align}

\subsection{Estimates for $I(x,t)$}
Without loss of generality we may take $x=0$ and assume that $\qpo$ is parallel to $\R^n$.
We further split our integral, so to compare the support of $\mu$ with the plane containing $0$. In this way, we may be able to use the $\beta$ numbers. 

\begin{align*}
    I(x,t) & \leq t^{-n} \av{ \int_{B(x,t)} K_t(x-y) - K_t(x-\poxo(y))\, d\mu(y) } 
     + t^{-n} \av{ \int_{B(x,t)} K(x- \pixo(y)) - K(x-\poxo(y))}  \\
    & := I_1(x,t) + I_2(x,t).
\end{align*}


\subsubsection{Estimates for $I_1$}
We make one more splitting: recall that $L_Q$ is the minimising plane of $\alpha(Q)$ and $L_Q^0$ is its translate containing $0$. Using \eqref{e:grad-est}, we write
\begin{align*}
    I_1(x,t) & \lesssim \frac{1}{t^n} \int_{B(0,t)} |y- \Pi_{L_Q}(y)|t^{-1} \, d\mu(y) 
     +\frac{1}{t^n} \int_{B(0,t)} |\Pi_{L_Q}(y) - \poxo(y)|t^{-1} \, d\mu(y) 
     =: I_{1,1} + I_{1,2}.
\end{align*}
Now, recalling the definitions in Subsection \ref{subsection:betaalpha}, 
\begin{align*}
    I_{1,1} = \beta_{\mu}^{1,n}(Q, L_Q) \lesssim \alpha_\mu(Q).
\end{align*}
On the other hand, we see that
\begin{align*}
    I_{1,2} \lesssim_{c_0} \frac{\dist(0, L_Q)}{\ell(Q)}.
\end{align*}

\subsubsection{Estimates for $I_2$}
If $y \in \R^d$, let $\theta_y$ be the angle between the plane $\R^n$ and the line segment $[0,y]$. Then, 
\begin{align} \label{eq:diffpixopoxo}
|\poxo(y) - \pixo(y)| & = |\pixo(y)|(1- \cos(\theta_y))
 = |\pixo(y)| \left( \sin(\theta_y/2)^2 \right)\nonumber \\
& \leq |\pixo(y)| \left( \sin(\theta_y/2) \right) 
= |\pixo(y)| \frac{\dist(y, L_Q^0)}{|\pixo(y)|} \nonumber \\
&= \dist(y, L_Q) + \dist(0, L_Q). 
\end{align}
Thus, as before, 
$
    I_2 \lesssim_{c_0} \alpha_\mu(Q) + \frac{\dist(0, L_Q)}{\ell(Q)}. 
$
All in all, we have 
$
    I \lesssim_{c_0} \alpha_\mu(Q) + \frac{\dist(0, L_Q)}{\ell(Q)}.
$
Hence for general $x \in Q$, we have that 
\begin{align*}
    I \lesssim_{c_0} \alpha_\mu (Q) + \frac{\dist(x, L_Q)}{\ell(Q)}.
\end{align*}
Recall Lemma \ref{lemma:tolsa2008} and Theorem \ref{theorem:tolsaalpha}. Then we see that
\begin{align*}
    & \left(\sum_{Q \in \mucubes(R)} \int_Q \int_{\ell(Q)}^{2\ell(Q)} I(x, t)^2\, d\mu(x) \dt\right)^{\frac{1}{2}} \\
    & \lesssim \left(\sum_{Q \in \mucubes(R)} \int_Q \int_{\ell(Q)}^{2\ell(Q)} \left(\alpha_\mu (Q) + \frac{\dist(x, L_Q)}{\ell(Q)}\right)^2 \, d\mu(x) \, \dt\right)^{\frac{1}{2}}\\
    & \lesssim \left(\sum_{Q \in \mucubes(R)} \alpha_\mu(Q)^2 \mu(Q)\right)^{\frac{1}{2}} \\
    & \enskip \enskip +\left( \sum_{Q \in \mucubes(R)} \int_Q \left(\frac{\dist(x, L_Q)}{\ell(Q)} \right)^2\, d\mu(x)\right)^{\frac{1}{2}} 
     \lesssim_{c_0} \mu(R)^{\frac{1}{2}}.
\end{align*}

\subsection{Estimates for $II(x,t)$}\label{s:est-II-Id}
For this estimate, we let $\Omega= Id$. We will comment later on the estimate for the case $\Omega \neq Id$ and $d=2, n=1$. 
Recall that (with $\Omega = Id$), 
\begin{align*}
II(x,t) = t^{-n} \left| \int_{B(x,t)} t^{-1} \left(\pix(y)-x \right) \, d \mu(y) \right|. 
\end{align*}
Due to the property \eqref{eq:06_30} of the circular projection, we notice that
\begin{align}\label{e:IIa}
    \int_{B(x,t)} \ps{\pix(y) - x} \, d\mu(y) & = \int_{\pix\left(B_d(x,t)\right)} \ps{y-x} \, d \pix[\mu] (y) 
     =\int_{B_n(x,t)} \ps{y -x} \, d \pix[\mu] (y).
\end{align}

For $y \in L_Q^x$, we may write 
\begin{align*}
    \int_{B_n(x,t)} \ps{y-x}\,  d \pix[\mu](y) = \sum_{i=1}^d \, \left(R_{L_Q^x}^{-1} e_i\right)  \, \int_{B_n(x,t)} \, (y-x) \cdot R_{L_Q^x}^{-1}(e_i)\, d\pix[\mu](y),
\end{align*}
where here $\{e_i\}_{i=1}^d$ is the standard orthonormal basis of $\R^d$. Notice that for $i=n+1, ..., d$, and for $y \in \qpx$, $(x-y) \cdot R_{L_Q^x}^{-1}(e_i) = 0$. 
For each $i \in \{1,...,d\}$ define
\begin{align*}
    g_i^{0}: \R^n \to \R
\end{align*}
by 
\begin{align*}
    g_i^{0} (y) := y \cdot e_i \chara_{\B_n}(y).
\end{align*}
Notice that 
\begin{align} \label{eq:supnormgio}
    \| \gio\|_\infty \leq 1.
\end{align}
We then see that
\begin{align*}
 & \int_{B_n(x,t)} \ps{\frac{y-x}{t}} \, d \pix[\mu] (y) \\
 & = \sum_{i=1}^n \left(R_{L_Q^x}^{-1} e_i\right)\int t^{-1} (x-y)\cdot R_{L_Q^x}^{-1}(e_i) \, d\pix[\mu](y) \\
 & = \sum_{i=1}^n \left(R_{L_Q^x}^{-1} e_i\right)\, t^{-1} \int_{B_n(x,t)} (R_{L_Q^x}(y-x)) \cdot e_i \, d\pix[\mu](y) \\
 & = \sum_{i=1}^n \left(R_{L_Q^x}^{-1} e_i\right) \, t^{-1} \int \gio(R_{L_Q^x}(y-x))  \, d\pix[\mu](y)
\end{align*}

Recall now Subsection \ref{subsection:wavelets}.
For each $i \in \{1,...,n\}$, we may decompose  $\gio$ through the wavelets basis $\{\wavi\}$. That is, we write $\gio$ as
\begin{align*}
    \gio(y)= \sum_{I \in \ncubes} a_I \wavi(y).
\end{align*}
Thus, 
\begin{align*}
    & \int \gio\left(R_{L_Q^x}\quotyx\right)  \, d\pix[\mu](y) \\
    &= \sum_{I \in \ncubes} a_I \int \wavi\left(\rot{x}{Q}\quotyx \right) \, d \pix[\mu](y).
\end{align*}
Moreover, 
\begin{align} \label{eq:circpush}
\int \wavi\ps{\rot{x}{Q}\quotyx}\, d \pix[\mu](y) = \int \wavi \ps{\frac{\pis_{\R^n} \left(R_{L_Q^x}(y-x)\right)}{t}} \, d \mu(y).
\end{align}

The following lemma gives us estimates on the coefficients $a_I$. Recall that 
\begin{align*}
    a_I = \langle \gio, \wavi \rangle_{L^2} = \int_{\R^n} \gio(y) \, \wavi(y) \, dy. 
\end{align*}
\begin{lemma} \label{lemma:wavcoeff}
Let $a_I$, $\wavi$ as above. Then, 
\begin{enumerate}
    \item If $5I \cap \partial \B= \emptyset$, then $a_I = 0$.
    \item If $\ell(I) \lesssim 1$, $|a_I| \lesssim \ell(I)^{n/2}$.
    \item If $\ell(I) \gtrsim 1$, $|a_I| \lesssim \ell(I)^{-n/2-1}$.
\end{enumerate}
\end{lemma}
\begin{proof} \
\begin{enumerate}
\item If $5I\subset \left( \overline \B\right)^c$, then $\wavi$ and $\gio$ have disjoint support and thus $a_I = 0$. Suppose that 
\begin{align} \label{eq:06_31}
    5I \subset \B \mbox{ but that } 5I \cap \partial \B= \emptyset.
\end{align} 
Then by Fubini's theorem, 
\begin{align*}
    \int_{\R^n} \gio(y) \wavi \, dy & = \int_{\B} y_i \, \wavi(y)\, dy\\
    & \eqt{\eqref{eq:06_31}} \int_{\R^n} y_i \, \wavi(y) \, dy \\
    & = \int_{\R^n} y_i \prod_{j=1}^n \wavi^j(y_j) \, dy_1 \cdots dy_n \\
    & = \left( \prod_{\substack{j=1\\ j \neq i}}^n \int_{\R} \wavi^j(y_j) \, dy_j \right) \left( \int_{\R} y_i \wavi^i (y_i)\, dy_i\right)\\
    & = 0
\end{align*}
since for each $i\in \{1,...,n\}$, $\wavi^i$ has zero mean. 
    \item We see that, by \eqref{eq:supnormgio} and recalling that $\|\wavi\|_{L^2} =1$ for each $I$,
    \begin{align*}
        \left| \int_{\R^n} \wavi(y) \, \gio(y) \, dy \right| & \leq \int |\wavi(y)|\, dy 
         \leq \ell(I)^{\frac{n}{2}} \|\wavi\|_{L^2}
        = \ell(I)^{\frac{n}{2}}.
    \end{align*}
    \item We use the fact that 
    \begin{align*}
        \int y \cdot e_i \chara_{\B} (y) \, dy = 0.
    \end{align*}
    Then we see that
    \begin{align*}
       \left| \int \gio(y)\wavi(y) \, dy \right| & = \left| \int \gio(y)(\wavi(y) - \wavi(0))\, dy \right| 
        \leq \int_{\B} |\gio(y)| |\wavi(y) - \wavi(0)| \, dy 
        \leq \int_{\B} \|\nabla \wavi\|_\infty |y| \, dy\\
       & \lesssim \|\nabla \wavi\|_\infty 
        \lesssim \ell(I)^{-\frac{n}{2} -1}
    \end{align*}
\end{enumerate}
\end{proof}

Now, notice that 
\begin{align} \label{eq:sptwavi}
    \spt\left(\wavi\left(\frac{\pis_{\R^n}\ps{\rot{x}{Q}(y-x)}}{t}\right) \right) \subset (\rot{x}{Q}\circ \pisp)^{-1}\left(t \cdot 5I\right) + x.
\end{align}

We see that if an $n$-cube satisfies either 
\begin{align} \label{eq:ns_ur_cWAV1}
    5I \cap \partial \B = \emptyset,
\end{align}
or
\begin{align} \label{eq:ns_ur_cWAV2}
    \left[(\rot{x}{Q}\circ \pisp)^{-1}\left(t \cdot 5I\right) + x \right] \cap \spt(\mu)= \emptyset,
\end{align}
then $I$ is negligible for our calculations. In the first case, \eqref{eq:ns_ur_cWAV1}, the wavelet coefficients $a_I$ will vanish identically, as from Lemma \ref{lemma:wavcoeff}. In the second case, \eqref{eq:ns_ur_cWAV2}, the integrand on the right hand side of \eqref{eq:circpush} and the measure $\mu$ will have disjoint support.

\begin{notation}
We say that $I \in \ncubes$ \textit{belongs to the subfamily of non negligible cubes}
$\dN\dG(x,t)$ if it satisfies the two following conditions.
\begin{align}
      5I \cap \partial \B \neq \emptyset \enskip \mbox{  and  } \enskip
       \left[(\rot{x}{Q}\circ \pisp)^{-1}\left(t \cdot 5I\right) + x \right] \cap \spt(\mu) \neq \emptyset
\end{align}
\end{notation}

\begin{remark}
From now on we will only consider $n$-cubes $I \in \dN\dG(x,t)$. Moreover, following the general strategy laid out in \cite{tolsa2014}, we will distinguish between $n$-cubes with small side length and $n$-cubes with large side length. 
\end{remark}

\begin{notation}
Let $0< \eta < 1$. We set 
\begin{align}
    & \dB\dN\dG(x,t) := \left\{ I \in \ncubes \, |\, I \in \dN\dG(x,t) \, \mbox{ and } \ell(I) \geq \eta \right\} \label{eq:nonegbig}\\
    & \dS\dN\dG(x,t) :=\left\{ I \in \ncubes \, |\, I \in \dN\dG(x,t) \, \mbox{ and } \ell(I) < \eta \right\} \label{eq:nonegsmall}\\
\end{align}
\end{notation}

\subsection{$II(x, t)$: estimates when $I \in \bigcubes(x,t)$}
Without loss of generality, we may let $x=0$ and assume that $\qpo$ is parallel to $\R^n$. 

Pick a $P = P(I) \in \mucubes$ so that
\begin{align*}
    \left(\pixo\right)^{-1}(t \cdot 5I) \subset 3 P(I), 
\end{align*}
and so that 
\begin{align*}
    \ell(P(I)) \sim \ell(Q) \ell(I). 
\end{align*}
Consider a smooth cut off function $\chi_P$ so that $\chara_{3P} \leq \chi_P \leq \chara_{B_P}$ and such that 
\begin{align} \label{eq:06_40}
    \|\nabla \chi_P \|_\infty \lesssim \frac{1}{\ell(P)}.
\end{align}
Then, for each $i \in \{1,...,n\}$, 
\begin{align*}
    e_i \, \int \gio \quoty \, d \pixo[\mu](y) & = e_i \sum_{I \in \ncubes} a_I\, \int \wavi\left(\frac{\pixo(y)}{t}\right)\, d\mu(y)\\
    & = \sum_{I \in \ncubes} e_i \, a_I \int \chi_P(y)  \wavi\left(\frac{ \pixo(y)}{t}\right)\, d\mu(y).
\end{align*}
Let $c_P$ be the constant which infimises $\alpha_\mu(P)$. As done previously, we want to compare the measure $\mu$, its push forward  through the circular projection, its push forward through the orthogonal projection and the $n$-dimensional Hausdorff measure on the best approximating plane and its translate: we write
\begin{align*}
& \left| \int \chi_P (y)\, \wavi \ps{\quotp} \, d\mu(y) \right| \\
& \leq \left| \int\chi_P(y)\, \left( \wavi \ps{\quotp} - \wavi\ps{\quotpor} \right) \, d\mu(y) \right| \\
& + \left| \int \chi_P(y) \, \wavi\ps{\quotpor} \, d\left( \mu - c_P \hn_{\qpo} \right)(y) \right|\\
& + \left| \int \chi_P(y) \wavi\ps{\quotpor} \, c_P \, d \hn_{\qpo} (y) \right| \\
& =: II_1 + II_2 + II_3.
\end{align*}

We want to obtain the following estimate, which resembles Lemma 5.3 in \cite{tolsa2014}.
\begin{lemma}
Let $I \in \bigcubes(x, t)$. Take $P$ as above. Then
\begin{align} \label{eq:ns_ur_cB1B2B3}
    II_1+ II_2+ II_3 \lesssim \left( \frac{\ell(Q)}{\ell(P)}\right)^{n/2} \left(\frac{\dist(0, L_Q)}{\ell(P)} +  \sum_{\substack{ S \in \mucubes\\Q \subset S \subset P}} \alpha_\mu(2S)\right) \ell(P)^n.
\end{align}
\end{lemma}

\begin{proof}
We arrived at a point where the quantities which we need to bound do not depend on the kernel which we started with, that is, the center of mass. We will solely use the properties of the wavelets decomposition. Thus the proof is, almost \textit{verbatim}, the proof of Lemma 5.3 in \cite{tolsa2014}.

We include it for the sake of completeness. 

\textbf{Estimates for $II_1$.}
We see that 
\begin{align*}
    II_1 \leq  \frac{\|\nabla \wavi\|_\infty}{t} \int_{B_P} \left| \pixo(y)- \poxo(y)\right| \, d \mu(y).
\end{align*}
Now, as in \eqref{eq:diffpixopoxo}, we have that
\begin{align*}
    \left| \pixo(y)- \poxo(y)\right| \lesssim \dist(y, \qpo).
\end{align*}
Let $y \in B_P \cap \spt(\mu)$. Let $L_P$ denote the $n$-plane which infimises $\alpha_\mu(P)$. Let $q_P \in L_P$ be so that
\begin{align*}
\dist(y, L_P) = \dist(y,q_P).
\end{align*}
Notice that, because $\alpha_\mu(1000Q) \leq \delta^2$,  $q_P \in B_P$ and $L_Q \cap B_P \neq \emptyset$. 
Then
\begin{align*}
    \inf_{q_Q \in L_Q} \dist(y, q_Q) & \leq \dist(y, q_P) + \inf_{q_Q \in L_Q \cap B_P} \dist(q_P, q_Q) \\
    & \leq \dist(y, L_P) + \dist_H\left(L_P \cap B_P, L_Q \cap B_P \right).
\end{align*}
Thus 
\begin{align*}
    \dist(y, \qpo) \leq \dist(0, L_Q) + \dist(y, L_P) + \dist_H\left(L_P \cap B_P, L_Q \cap B_P\right).
\end{align*}
As in \eqref{eq:s_ur_c_A21b}, by Lemma 5.2 in \cite{tolsa2008}, and recalling that $P(I) \supset Q$, we see that
\begin{align} \label{eq:ns_ur_cB1a}
    \dist_H\left(L_P \cap B_P, L_Q \cap B_P\right) \leq  \sum_{\substack{ S \in \mucubes\\Q \subset S \subset P}} \alpha_\mu(S) \ell(S).
\end{align}
Hence 
\begin{align*}
   & \int_{B_P} \left| \pixo(y)- \poxo(y)\right| \, d \mu(y)\\
   &  \lesssim \int_{B_P} 
    \dist(0, L_Q) d\mu(y)  + \int_{B_P} \dist(y, L_P) \, d\mu(y) + \int_{B_P} \sum_{\substack{ S \in \mucubes\\Q \subset S \subset P}} \alpha_\mu(S) \ell(S) \, d\mu(y) \\
    & \lesssim \dist(0, L_Q) \mu(P) + \int_{B_P} \dist(y, L_P) \, d\mu(y) + \ell(P)^{n+1} \sum_{\substack{ S \in \mucubes\\Q \subset S \subset P}} \alpha_\mu(S) \\
    & \lesssim \dist(0, L_Q) \mu(P) + \ell(P)^{n+1} \sum_{\substack{ S \in \mucubes\\Q \subset S \subset P}} \alpha_\mu(2S).
\end{align*}
The last inequality is due to Remark 3.3 in \cite{tolsa2008}.

All in all, we get that
\begin{align*}
II_1 \lesssim \|\nabla \wavi\|_\infty\frac{\ell(P)}{\ell(Q)} \left(\frac{\dist(0, L_Q)}{\ell(P)} +  \sum_{\substack{ S \in \mucubes\\Q \subset S \subset P}} \alpha_\mu(2S)\right) \ell(P)^n.
\end{align*}
Now, because $\ell(P)\sim \ell(Q) \ell(I)$, and by the bound $\|\nabla \wavi\|_\infty \lesssim \frac{1}{\ell(I)^{1+ n/2}}$, we see that
\begin{align*}
    \|\nabla \wavi\|_\infty\frac{\ell(P)}{\ell(Q)} \lesssim \left( \frac{\ell(Q)}{\ell(P)}\right)^{n/2}.
\end{align*}
Thus 
\begin{align} \label{eq:ns_ur_cB1}
    II_1 \lesssim \left( \frac{\ell(Q)}{\ell(P)}\right)^{n/2} \left(\frac{\dist(0, L_Q)}{\ell(P)} +  \sum_{\substack{ S \in \mucubes\\Q \subset S \subset P}} \alpha_\mu(2S)\right) \ell(P)^n.
\end{align}

\textbf{Estimates for $II_2$}
Once again, we add and subtract the quantities which we are interested in comparing. 
\begin{align*}
II_2 & = \left| \int \chi_P(y) \, \wavi\ps{\quotpor} \, d\left( \mu - c_P \hn_{\qpo} \right)(y) \right|\\
& \leq \left| \int \chi_P(y) \, \wavi\ps{\quotpor} \, d\left( \mu - c_P \hn_{L_P} \right)(y) \right| \\
& \enskip \enskip + \left|\int \chi_P(y) \, \wavi\ps{\quotpor} \, d\left( c_P\hn_{L_P} - c_P \hn_{\qpo} \right)(y) \right|\\
& =: II_{2,1} + II_{2,2}.
\end{align*}
Notice that, recalling the properties of the wavelets $\wavi$ (see \eqref{eq:02_waviInf} and \eqref{eq:02_waviDifInf}), 
\begin{align*}
\left|\nabla \left( \chi_P(y) \, \wavi \ps{\quotpor}\right) \right| & \leqt{\eqref{eq:06_40}} \frac{1}{\ell(P)} \|\wavi\|_\infty + \left| \nabla \wavi\ps{\quotpor} \right|\\
& \lesssim \frac{1}{\ell(P)} \left( \frac{\ell(Q)}{\ell(P)}\right)^{n/2} + \left(\frac{\ell(Q)}{\ell(P)}\right)^{n/2 +1} \frac{1}{\ell(Q)}\\
& = \frac{1}{\ell(Q)} \left( \frac{\ell(Q)}{\ell(P)}\right)^{n/2+1}.
\end{align*}
Hence
\begin{align*}
II_{2,1} & \lesssim \left\| \nabla \left( \chi_P \, \wavi\ps{\frac{\poxo(\cdot)}{t}} \right) \right\|_\infty \alpha_\mu(P)\ell(P)^{n+1} \\
& \lesssim \frac{1}{\ell(Q)} \left( \frac{\ell(Q)}{\ell(P)}\right)^{n/2+1} \alpha_\mu(P)\ell(P)^{n+1}\\
& = \left( \frac{\ell(Q)}{\ell(P)}\right)^{n/2}\alpha_\mu(P)\ell(P)^{n}.
\end{align*}

Similarly, and recalling \eqref{eq:ns_ur_cB1a} and the subsequent discussion,
\begin{align*}
    II_{2,2} & \lesssim \frac{1}{\ell(Q)} \left( \frac{\ell(Q)}{\ell(P)}\right)^{n/2+1} \dist_H \left( B_P \cap L_P, B_P \cap \qpo \right) \ell(P)^n\\
    & \left( \frac{\ell(Q)}{\ell(P)}\right)^{n/2} \left(\frac{\dist(0, L_Q)}{\ell(P)} +  \sum_{\substack{ S \in \mucubes\\Q \subset S \subset P}} \alpha_\mu(2S)\right) \ell(P)^n.
\end{align*}
Thus
\begin{align} \label{eq:ns_ur_cB2}
    II_2 \lesssim \left( \frac{\ell(Q)}{\ell(P)}\right)^{n/2} \left(\frac{\dist(0, L_Q)}{\ell(P)} +  \sum_{\substack{ S \in \mucubes\\Q \subset S \subset P}} \alpha_\mu(2S)\right) \ell(P)^n.
\end{align}

\textbf{Estimates for $II_3$}
Because $\wavi$ has zero mean, we see that
\begin{align} \label{eq:ns_ur_c_B3}
    II_3 = 0.
\end{align}

Putting together \eqref{eq:ns_ur_cB1}, \eqref{eq:ns_ur_cB2} \and \eqref{eq:ns_ur_c_B3}, we get that, when $I \in \dB\dN\dG(x, t)$, 
\begin{align} \label{eq:ns_ur_cB1B2B3}
    II_1+ II_2+ II_3 \lesssim \left( \frac{\ell(Q)}{\ell(P)}\right)^{n/2} \left(\frac{\dist(0, L_Q)}{\ell(P)} +  \sum_{\substack{ S \in \mucubes\\Q \subset S \subset P}} \alpha_\mu(2S)\right) \ell(P)^n.
\end{align}
\end{proof}

\subsection{$II(x, t)$: estimates when $I \in \smallcubes(x,t)$} \

\begin{remark}
What we mentioned above holds here, too: the quantity we are estimating do not depend on the kernel, but rather on the properties of the wavelets decomposition. Thus this subsection will resemble very closely Subsection 5.3 in \cite{tolsa2014}: in the stopping time argument below, we stop only depending on angles, just as in \cite{tolsa2014}; the subsequent estimates will follow as they follow in \cite{tolsa2014}. As a matter of fact, the reader might end up under the impression of doing shopping at `The Other Mathematicians' Tools Warehouse'. 
\end{remark}

Notice that, if before we had that $Q \subset P(I)$, because $\ell(I) \leq  \eta$, now we have the opposite containment (up to a constant):
\begin{align*}
P(I) \subset C Q.
\end{align*}
Choosing $\eta >0$ appropriately, we may pick $C=1000$. 
For a fixed $Q \in \mucubes$, we introduce a stopping time condition on the $P \in \mucubes$: let $P \in \good$ if the following two conditions hold true.
\begin{enumerate}
    \item We have
    \begin{align} \label{eq:ns_ur_cBsmall1}
    P \subset 1000Q.
    \end{align}
    \item We have that
    \begin{align} \label{eq:ns_ur_cBsmall2}
        \sum_{\substack{S \in \mucubes\\ P \subset S \subset 1000Q}} \alpha_\mu(100S) \leq \delta.
    \end{align}
\end{enumerate}

Let now $\term$ be the subfamily of cubes in $\mucubes \setminus \good$ (these are `bad' cubes!) which are maximal with respect to inclusion. It is a well known issue that adjacent cubes belonging to $\term$ may have wildly different size. This can cause troubles; we resort to a well known smoothing procedure, which will output a family of maximal `bad' cubes which have comparable size if they are close. We define
\begin{align}
    & \mbox{for } y \in \R^d, \enskip d(y):= \inf_{P \in \good} \left(\ell(P) + \dist(y, P)\right), \\
    & \mbox{for } z \in \R^n, \enskip D(z) := \inf_{y \in (\pisp)^{-1}\{z\} } d(y).
\end{align}

\begin{lemma}[{\cite{tolsa2014}, Lemma 5.6}]
The function $d$ is 1-Lipschitz and the function $D$ is 3-Lipschitz. Moreover, if the $\delta$ in \eqref{eq:ns_ur_cdelta} is chosen small enough, $d(y)< \infty $ ans $D(z)< \infty$.
\end{lemma}
\begin{proof}
See Lemma 5.6 in \cite{tolsa2014} and the comment immediately after the proof.
\end{proof}

Now set 
\begin{align*}
    \dF:=\left\{I \in \ncubes \, |\, t  \ell(I) \leq \frac{1}{5000} \inf_{z \in t \cdot I} D(z) \right\}.
\end{align*}
Moreover, we let $\reg$ to be the subfamily of $\dF$ of maximal cubes. That is, 
\begin{align*}
    \reg:= \left\{ I \in \dF \, |\, \mbox{if } J \in \dF \mbox{ and } J \cap I \neq \emptyset, \mbox{ then } I \supset J\right\}.
\end{align*}

\begin{lemma}
The cubes in $\reg$ are pairwise disjoint. Moreover, if $I, J \in \reg$ and
\begin{align*}
20J \cap 20I \neq \emptyset,
\end{align*}
then 
\begin{align*}
    \ell(I) \sim \ell(J).
\end{align*}
\end{lemma}
\begin{proof}
This follows as in \cite{singularintegrals}, Lemma 8.7.
\end{proof}

We now define two subfamilies of cubes in $\smallcubes(x,t)$, which will need each one its own treatment. 

Let $I \in \tree(x, t) \subset \ncubes$ if 
\begin{align*}
    I \in \smallcubes \mbox{ and there is no } J \in \reg(x,t) \mbox{ so that  } J \supset I.
\end{align*}

Let $I \in \Stop(x,t)\subset \ncubes$ if
\begin{align*}
    I \in \smallcubes(x, t) \cap \reg(x, t).
\end{align*}

We write
\begin{align}
    & \sum_{I \in \smallcubes(x,t)} \frac{a_I}{t^n} \, \int \wavi\ps{\quotr}\, d\mu(y)\nonumber\\
    &= \sum_{I \in \tree(x,t)} \frac{a_I}{t^n} \, \int \wavi\ps{\quotr}\, d\mu(y)\nonumber \\
    & \enskip \enskip + \sum_{I \in \Stop(x,t)} \sum_{\substack{J \in \smallcubes(x,t)\\ J \subset I}} \frac{a_J}{t^n} \int \wavj \left( \frac{\pisp\ps{R_{L_Q^x}(y-x)}}{t} \right) \, d\mu(y). \label{eq:splittree}
\end{align}

\subsubsection{$II(x, t)$: estimates when $I \in \tree(x,t)$}

\begin{lemma} \label{lemma:existenceP}
Let $I \in \tree(x,t)$. There exists a $P:=P(I) \in \mucubes$ with
\begin{align*}
    \ell(P) \sim \ell(I) t \,\, \mbox{ and } \spt(\mu) \cap \left(x + \left(\pix\right)^{-1}\left(t \cdot 5I\right)\right) \subset 3P.
\end{align*}
\end{lemma}
\begin{proof}
See Lemma 5.7 in \cite{tolsa2014} and the proof of it. Notice that the hypothesis
\begin{align*}
    5I \cap (\partial B_n(0,1) \cup \partial B_n(0,2)) \neq \emptyset
\end{align*}
is substituted with 
\begin{align*}
    5I \cap \partial B_n(0,1) \neq \emptyset.
\end{align*}
\end{proof}

\begin{lemma}
Let $I \in \tree(x,t)$ and let $P:=P(I)$ be a $\mu$-cube satisfying the conclusions of Lemma \ref{lemma:existenceP}. Then
\begin{align}
& \left| \int \wavi\left(\frac{\pisp(\rot{x}{Q}(y-x))}{t} \right) \, d\mu(y) \right| 
 \lesssim 
\ps{\frac{\ell(Q)}{\ell(P)}}^{n/2} \left( \sum_{\substack{S \in \mucubes \\ P \subset S \subset Q}} \alpha(S) + \frac{\dist(x, L_Q)}{\ell(Q)}\right) \ell(P)^n. \label{eq:06_51}
\end{align}
\end{lemma}

\begin{proof}
The proof of Lemma 5.8 in \cite{tolsa2014} can easily be adapted to our current situation; we include it for the sake of completeness. The idea, as in previous lemmata, is to bound the wavelet with the sum of angles between $n$-planes, or $\alpha_\mu$ numbers, which are under control because the measure $\mu$ is uniformly rectifiable. 

Without loss of generality we may take $x=0$ and assume that $L_Q$ is parallel to $\R^n$. Denoting by $L_P$ the plane which minimises $\alpha_\mu(P)$,  let $q_P \in B_P \cap \spt(\mu)$ be so that 
\begin{align} \label{eq:06_50}
    \dist(q_P, L_P) \lesssim \alpha(P) \ell(P).
\end{align}
Such a point exists: for some constant $c^*>0$, set 
\begin{align*}
    & A(c*) =A  \\
    & := \left\{ y \in B_P  \cap \spt(\mu) \, |\, \dist(y, L_P)\ell(P)^{-1} < c^* \ell(P)^{-n} \int_{B(z_P, 2 \ell(P))}  \frac{\dist(y, L_P)}{\ell(P)} \, d\mu(y) \, \right\}.
\end{align*}
Then, by Chebyshev inequality, we see that
\begin{align*}
    \mu(B_P \setminus A) & \leq \frac{1}{c^* \, \beta_{\mu, 1}(P) } \int_{B_P} \frac{\dist(y, L_P)}{\ell(P)}\, d\mu(y) \\
    & \leq \frac{\ell(P)}{c^*}.
 \end{align*}
 Choosing $c^*$ large enough (not depending on $P$), we see that $A \neq \emptyset$. Finally, recalling \eqref{eq:02_50}, we verify the existence of a point $y_P \in \spt(\mu) \cap B_P$ so that \eqref{eq:06_50} holds.
 
Recall that $\qpo$ is the plane parallel to $L_Q$ but containing $0$. Now, denote by $\tilde L_P$ the plane parallel to $\qpo$ (and thus to $\R^n$) which moreover contains $q_P$. As before, let $\chi_P$ be a smooth bump function with 
\begin{align} \label{eq:06_53}
    \chara_{B_P} \leq \chi_P \leq \chara_{3B_P} \enskip \mbox{ and }\enskip  \|\nabla \chi_P \|_\infty \lesssim \frac{1}{\ell(P)}.
\end{align}
Since $\alpha(P)$ is assumed to be very small, we have that 
\begin{align}
    \left(\pis\right)^{-1} (5I \cdot t) \cap \tilde L_P \subset B_P. \label{eq:06_56}
\end{align}
Recall that $\pis= \pis_\{\R^n\}$.
Set
\begin{align*}
    & \sigma_P := c_P \hn|_{L_P} \mbox{ and } \\
    & \tilde \sigma_P := c_P \hn|_{\tilde L_P}.
\end{align*}
We split the integral on the left hand side of \eqref{eq:06_51} so to compare the measure $\mu$ to $\sigma_P$, and then $\sigma_P$ to $\tilde \sigma_P$.
\begin{align*}
    \int \wavi \left( \frac{\pis(y)}{t}\right) \, d\mu(y) & = \int \chi_P(y) \wavi \left( \frac{\pis(y)}{t}\right) \, d\mu(y) \\
    & = \int \chi_P(y) \, \wavi\ps{\frac{\pis(y)}{t}} \, d \left( \mu - \sigma_P\right)(y) \, \\
    & \enskip \enskip + \int \chi_P(y) \, \wavi \ps{\frac{\pis(y)}{t}} \, d \ps{\sigma_P -\tilde \sigma}(y)\, \\
    & \enskip \enskip + \int \chi_P(y) \, \wavi\ps{\frac{\pis(y)}{t}} \, d \tilde \sigma_P(y) \\
    & =: A_1 +A_2+ A_3.
\end{align*}
\begin{enumerate}
    \item \textit{Estimates for $A_1$.} First, notice that since $I$ has small size and it intersects the boundary of the unit ball on the plane, and moreover, $Q$ lies close to $\qpo$ (that is, $\alpha_\mu(1000Q) \leq \delta$), then $B_P$ is far from $(\pis)^{-1}(\{0\})$. Thus we have 
    \begin{align}
        \|\nabla \pis\|_\infty \lesssim 1 \mbox{ on } B_P.
    \end{align}
    Moreover, recall that $\ell(P) \sim \ell(I) \, t \sim \ell(I) \, \ell(Q)$. Hence this together with the properties of the Debauchies wavelets give us
    \begin{align}
        |A_1 |& \lesssim \left\| \nabla  \left(\chi_P(y) \, \wavi\ps{\frac{\pis(y)}{t}}\right) \right\|_\infty \, \alpha(P) \ell(P)^{n+1}\nonumber \\
        & \lesssimt{\eqref{eq:06_53}, \eqref{eq:02_waviInf}, \eqref{eq:02_waviDifInf}} \left( \frac{1}{\ell(P) \, \ell(I)^{n/2}}+ \frac{1}{\ell(I)^{n/2 +1}} \frac{1}{t}\right) \, \alpha(P) \ell(P)^{n+1} \nonumber \\
        & \sim \ps{\frac{\ell(Q)}{\ell(P)}}^{\frac{n}{2}} \alpha(P) \ell(P)^n. \label{eq:06_55}
    \end{align}
    \item \textit{Estimated for $A_2$.} As above, we write
    \begin{align}
        |A_2| & \lesssim \left( \frac{1}{\ell(P) \, \ell(I)^{n/2}}+ \frac{1}{\ell(I)^{n/2 +1}} \frac{1}{t}\right) \, \dist_{B_P}(\sigma_P, \tilde \sigma_P) \nonumber\\
        & \lesssim \left( \frac{1}{\ell(P) \, \ell(I)^{n/2}}+ \frac{1}{\ell(I)^{n/2 +1}} \frac{1}{t}\right) \, \ell(P)^n \, \dist_H(L_P \cap B_P, \tilde L_P \cap B_P). \label{eq:06_54}
    \end{align}
    Now, from \cite{tolsa2014}, Lemma 5.2, and using here our choice of $q_P$, i.e. \eqref{eq:06_50}, we see that 
    \begin{align*}
        \frac{1}{\ell(P)} \, \dist_H(L_P \cap B_P, \tilde L_P \cap B_P) \lesssim \sum_{P \subset S \subset Q} \alpha(S).
    \end{align*} 
    Hence, together  with \eqref{eq:06_54} and as in \eqref{eq:06_55}, we obtain
    \begin{align}
        |A_2| \lesssim \ps{\frac{\ell(Q)}{\ell(P)}}^{\frac{n}{2}} \alpha(P) \ell(P)^n.
    \end{align}
    \item \textit{Estimates for $A_3$.} Let $B$ be a ball centered on $\qpo$ containing the support of $\wavi\ps{\frac{\cdot}{t}}$ and with $r(B)\lesssim \ell(P)$. Let $0 < \tilde c$ be a constant which will be chosen below. Recall that $\tilde L_P$ is the plane parallel to $L_Q$ but containing $q_P$. We want to compare $\tilde \sigma$ to the $n$-dimensional Hausdorff measure restricted to $\qpo$: we further split the integral $A_3$ in the following way.
    \begin{align*}
        |A_3| & \eqt{\eqref{eq:06_56}} \left| \int \wavi \ps{\frac{\pis(y)}{t}} \, d\tilde \sigma_P(y) \, \right| \\
        & = \left| \int \wavi \ps{\frac{y}{t}} \, d \pis [\tilde \sigma_P](y) \, \right|\\
        & \leq \left| \int \wavi \ps{\frac{y}{t}} \, d \pis [\tilde \sigma_P](y) - \tilde c \, c_P \, \int \wavi \ps{\frac{y}{t}}\, d \hn|_{\qpo}(y)\right| \\
        & \enskip \enskip + \tilde c\, c_P \left| \int \wavi \ps{\frac{y}{t}} \, d\hn|_{\qpo}(y)\right|.
    \end{align*}
    Notice that since $0 \in \qpo$, the second term on the right hand side equals to $0$. The first term can be bounded as in \eqref{eq:06_54} (recalling that $c_P \lesssim 1$):
    \begin{align*}
        |A_3| \lesssim \frac{1}{\ell(I)^{n/2+1} \, t} \, \dist_B (\pis[\hn|_{\tilde L_P}], \tilde c \, \hn|_{\qpo}).
    \end{align*}
    We need to bound $\dist_B (\pis[\hn|_{\tilde L_P}], \tilde c \, \hn|_{\qpo})$. 
    
     \begin{sublemma}
     With the notation as above, we have
     \begin{align*}
        & \dist_B (\pis[\hn|_{\tilde L_P}], \tilde c \, \hn|_{\qpo}) 
         \lesssim \left( \sum_{P \subset S \subset Q} \alpha(S) + \frac{\dist(0, L_Q)}{\ell(Q)} \right) \ell(P)^{n+1}.
    \end{align*}
    \end{sublemma}
    \begin{proof}[Proof of sublemma]
    This is done in \cite{tolsa2014}, see the proof of Lemma 5.8, the paragraph below it and Lemma 5.9. 
    \end{proof}
    
    This and the previous estimates for $A_1$ and $A_2$ give the desired result. 
\end{enumerate}
\end{proof}

Set now
\begin{align*}
    \uptree(x, t) := \left\{ P \in \mucubes \,|\, P=P(I) \mbox{ for } I \in \tree(x,t) \right\}.
\end{align*}

\begin{lemma}
Keep the notation as above. Then
\begin{align*}
    & \left| \sum_{I \in \tree(x,t)} \frac{a_I}{t^n} \int \wavi \left( \frac{\pisp\left(R_{L_Q^x}(y-x)\right)}{t} \right) \, d\mu(y)\right| \\
    & \lesssim \sum_{P \in  \uptree(x,t)} \left( \alpha (aP)  + \frac{\dist(x, L_Q)}{\ell(Q)}\right) \frac{\mu(P)}{\mu(Q)},
\end{align*}
for some absolute constant $a \geq 1$.
\end{lemma}
\begin{proof}
Again, one can easily adapt the proof of Lemmas 5.10 and 5.11 in \cite{tolsa2014} to the current situation.
\end{proof}

\begin{lemma} \label{lemma:ns_ur_cSNG1}
With notation as above, we have
\begin{align*}
& \sum_{\substack{Q \in \mucubes(R)\\ \alpha_\mu(1000Q)\leq \delta^2}} \int_Q \int_{\ell(Q)}^{2\ell(Q)} 
\left( \left| \sum_{I \in \tree(x,t)} \frac{a_I}{t^n} \int \wavi \left( \frac{\pisp\left(R_{L_Q^x}(y-x)\right)}{t} \right) \, d\mu(y)\right|^2 \right)\, \dt \, d\mu(x) \\
& \lesssim \mu(R).
\end{align*}
\end{lemma}
\begin{proof}
This follows as Lemma 5.12 in \cite{tolsa2014}.
\end{proof}
This concludes the estimates for $I \in \tree(x,t)$.

\subsubsection{$B$: estimates for $I \in \Stop(x,t)$}.
\begin{lemma} \label{lemma:ns_ur_cSNG2}
Keep the notation as above. Then
\begin{align*}
& \sum_{\substack{Q \in \mucubes(R)\\ \alpha_\mu(1000Q)\leq \delta^2}} \int_Q \int_{\ell(Q)}^{2\ell(Q)} \left| \sum_{I \in \Stop(x,t)} \sum_{\substack{J \in \smallcubes(x,t)\\ J \subset I}} \frac{a_J}{t^n} \int \wavj \left( \frac{\pisp\ps{R_{L_Q^x}(y-x)}}{t} \right) \, d\mu(y) \right|^2 \, \dt \, d\mu(x)\\
& \lesssim \mu(R).
\end{align*}
\end{lemma}
This follows as Lemmas 5.13 and 5.14 in \cite{tolsa2014}.

\subsection{Final estimates}
We have that
\begin{align*}
& \sum_{\substack{Q \in \mucubes(R)\\ \alpha_\mu(1000Q)\leq \delta^2}} \int_Q \int_{\ell(Q)}^{2\ell(Q)} II(x, t)^2 \, \dt \, d\mu(x)\\ 
& = \sum_{\substack{Q \in \mucubes(R)\\ \alpha_\mu(1000Q)\leq \delta^2}}  \int_Q \int_{\ell(Q)}^{2\ell(Q)} \left| \sum_{i=1}^d \sum_{I \in \dN\dG(x,t)} \frac{a_I}{t^n} \int \wavi^i \ps{\frac{\pisp\left(R_{L_Q^x}(y-x)\right)}{t}} \, d\mu(y)\, \right|^2 \, \dt d\mu(x) \\
& \leq \sum_{i=1}^d \sum_{\substack{Q \in \mucubes(R)\\ \alpha_\mu(1000Q)\leq \delta^2}}  \int_Q \int_{\ell(Q)}^{2\ell(Q)} \left|  \sum_{I \in \dN\dG(x,t)} \frac{a_I}{t^n} \int \wavi^i \ps{\frac{\pisp\left(R_{L_Q^x}(y-x)\right)}{t}} \, d\mu(y)\, \right|^2 \, \dt d\mu(x) \\
\end{align*}

For each $i \in \{1,...,d\}$, we write
\begin{align*}
& \sum_{\substack{Q \in \mucubes(R)\\ \alpha_\mu(1000Q)\leq \delta^2}}  \int_Q \int_{\ell(Q)}^{2\ell(Q)} \left|  \sum_{I \in \dN\dG(x,t)} \frac{a_I}{t^n} \int \wavi^i \ps{\frac{\pisp\left(R_{L_Q^x}(y-x)\right)}{t}} \, d\mu(y)\, \right|^2 \, \dt d\mu(x) \\
& \lesssimt{\eqref{eq:nonegbig}, \eqref{eq:nonegsmall}} \sum_{\substack{Q \in \mucubes(R)\\ \alpha_\mu(1000Q)\leq \delta^2}}  \int_Q \int_{\ell(Q)}^{2\ell(Q)} \left|  \sum_{I \in \bigcubes(x,t)} \frac{a_I}{t^n} \int \wavi^i \ps{\frac{\pisp\left(R_{L_Q^x}(y-x)\right)}{t}} \, d\mu(y)\, \right|^2 \, \dt d\mu(x) \\
& \enskip \enskip +  \sum_{\substack{Q \in \mucubes(R)\\ \alpha_\mu(1000Q)\leq \delta^2}}  \int_Q \int_{\ell(Q)}^{2\ell(Q)} \left|  \sum_{I \in \smallcubes(x,t)} \frac{a_I}{t^n} \int \wavi^i \ps{\frac{\pisp\left(R_{L_Q^x}(y-x)\right)}{t}} \, d\mu(y)\, \right|^2 \, \dt d\mu(x) 
\end{align*}
\begin{align*}
& \lesssimt{\eqref{eq:splittree}}  \sum_{\substack{Q \in \mucubes(R)\\ \alpha_\mu(1000Q)\leq \delta^2}}  \int_Q \int_{\ell(Q)}^{2\ell(Q)} \left|  \sum_{I \in \bigcubes(x,t)} \frac{a_I}{t^n} \int \wavi^i \ps{\frac{\pisp\left(R_{L_Q^x}(y-x)\right)}{t}} \, d\mu(y)\, \right|^2 \, \dt d\mu(x)\\
& \enskip \enskip + \sum_{\substack{Q \in \mucubes(R)\\ \alpha_\mu(1000Q)\leq \delta^2}}  \int_Q \int_{\ell(Q)}^{2\ell(Q)} \left|  \sum_{I \in \tree(x,t)} \frac{a_I}{t^n} \int \wavi^i \ps{\frac{\pisp\left(R_{L_Q^x}(y-x)\right)}{t}} \, d\mu(y)\, \right|^2 \, \dt d\mu(x) \\
& \enskip \enskip +   \sum_{\substack{Q \in \mucubes(R)\\ \alpha_\mu(1000Q)\leq \delta^2}}  \int_Q \int_{\ell(Q)}^{2\ell(Q)} \left|  \sum_{I \in \Stop(x,t)} \sum_{\substack{J \in \smallcubes(x,t)\\ J \subset I}} \frac{a_J}{t^n} \int \wavj^i \left( \frac{\pisp\ps{R_{L_Q^x}(y-x)}}{t} \right) \, d\mu(y)\right|^2 \, \dt d\mu(x)
\end{align*}
\begin{align*}
& \lesssim \sum_{\substack{Q \in \mucubes(R)\\ \alpha_\mu(1000Q)\leq \delta^2}}  \int_Q \int_{\ell(Q)}^{2\ell(Q)} \left|  \sum_{I \in \bigcubes(x,t)} \frac{a_I}{t^n} \int \wavi^i \ps{\frac{\pisp\left(R_{L_Q^x}(y-x)\right)}{t}} \, d\mu(y)\, \right|^2 \, \dt d\mu(x)\\
&\enskip \enskip + \mu(R),
\end{align*}
by Lemma \ref{lemma:ns_ur_cSNG1} and Lemma \ref{lemma:ns_ur_cSNG2}.

Moreover, by \eqref{eq:ns_ur_cB1B2B3}, we have that 
\begin{align*}
 & \left|  \sum_{I \in \bigcubes(x,t)} \frac{a_I}{t^n} \int \wavi^i \ps{\frac{\pisp\left(R_{L_Q^x}(y-x)\right)}{t}} \, d\mu(y)\, \right|^2 \\
 &\lesssim \sum_{I \in \bigcubes(x,t)} \frac{a_I}{t^n} \left( \frac{\ell(Q)}{\ell(P(I))}\right)^{n/2} \left(\frac{\dist(0, L_Q)}{\ell(P(I))} +  \sum_{\substack{ S \in \mucubes\\Q \subset S \subset P(I)}} \alpha_\mu(2S)\right) \ell(P(I))^n\\
 \end{align*}
 
 Arguing as in Lemma 5.4 and Lemma 5.5 in \cite{tolsa2014} one obtains the desired Carleson estimate. This concludes the proof of Proposition \ref{proposition:ns_ur_c}.
 
 \subsection{Estimates for $II(x,t)$ when $\Omega \neq Id$}
 Let us consider the estimate of $II(x,t)$ (as given in \eqref{e:I+II}). Recall that from Subsection \ref{s:est-II-Id} onward we assumed that $\Omega = Id$. Assume now that $d=2$, $n=1$ and that $\Omega$ and $\mu$  are as in the statement of Theorem \ref{t:UR-Omega}. 
 
 Recall that $II(x,t) = t^{-1} \av \int_B(x,t) K_t(\pixo(y) - x) \, d\mu(y)$. We proceed as in \eqref{e:IIa} to arrive at $$t^{-1} \int_{B_1(x,t)} K_t(y-x) \, d\pixo[\mu](y).$$
 Note that because $K(x) = |x| \Omega(x/|x|)$, $K(L_Q^x-x)$ is a line through the origin. This is not the case if $L_Q^x$ is an $n$-plane in $\R^d$ for $n \neq 1$ or $d \neq 2$. However, with the assumptions of Theorem \ref{t:UR-Omega}, this holds, and thus we can proceed as in Subsection \ref{s:est-II-Id} to define the appropriate function $g_1$ as in \eqref{e:g}, where this time $\{e_1, e_2\}$ will be taken to be the standard basis of $K(L_Q^x-x) \times K(L_Q^x-x)^\perp$. 
 The rest of the argument goes through unchanged. This gives Proposition \ref{p:UR-C-Omega}, and thus one direction of Theorem \ref{t:UR-Omega}.
\section{Rectifiability implies finiteness of square function}

Define the operator $C_\mu$ by
\begin{align*}
C_\mu(f)(x) := \left( \int_0^\infty \left|C_{f\mu}(x,t)\right|^2 \, \dt\right)^{\frac{1}{2}}.
\end{align*}
Here $x \in \R^d$.

Here we will prove the following. 
\begin{proposition}\label{p:rectif}
Let $\mu$ be a finite $n$-rectifiable measure on $\R^d$. Then 
\begin{align*}
    C_\mu(x) < \infty \enskip \mbox{ for } \mu \mbox{-almost every } x \in \spt(\mu).
\end{align*}
\end{proposition}

\begin{remark}
The same proof goes through almost verbatim for the case where $K(x)= |x|\Omega(x/|x|)$, and $\Omega$ is as in the statement of Theorem \ref{t:R-Omega}. Actually, the proof below would work in any dimensions $n,d$. However, it rests upon Proposition \ref{proposition:ns_ur_c}. Since we only have Proposition \ref{p:UR-C-Omega} for $n=1$ and $d=1$, we can prove Theorem \ref{t:R-Omega} only in this case. 
\end{remark}

To prove Proposition \ref{p:rectif}, we will show that if $\mu$ is a finite $n$-uniformly rectifiable measure on $\R^d$, and if $\nu$ is a Borel measure, then 
\begin{align} \label{eq:rec_finite1}
    \mu\ps{\ck{ x\in \R^d \, |\, C_\nu(x) > \lambda }} \leq \frac{\|\nu\|}{\lambda}
\end{align}
Now we let $\spt(\mu)$ be decomposed into a countable compact subsets, say $\{E_n\}$, and we let $\nu=\mu|_{E_n}$. Then, assuming \eqref{eq:rec_finite1}, 
\begin{align*}
\hn|_{E_n} \ps{\ck{ x \in \R^d \, |\, C_\mu(x) > \lambda}} \leq \frac{\|\mu\|}{\lambda}.
\end{align*}

\subsection{$L^2(\mu)$ boundedness of $C_\mu$}
This subsection will be devoted to proving the following proposition.
\begin{proposition} \label{proposition:L2}
Let $\mu$ be an $n$-uniformly rectifiable measure on $\R^d$ such that $\mu(\R^d)< \infty$. The operator $C_\mu$ is bounded from $L^2(\mu)$ to $L^2(\mu)$.
\end{proposition}

Set 
\begin{align*}
    \cmuk(f)(x) := \left( \int_{2^{-k-2}}^{2^{-k-1}} \left| \cmuef(x, t)\right|^2 \, \dt \right)^{1/2}.
\end{align*}
We then may write 
\begin{align*}
    C_\mu(f)(x)^2= \int_0^\infty |\cmuef(x,t)|^2\, \dt = \sum_{k \in \Z} \cmuk(f)(x)^2.
\end{align*}
We have that
\begin{align} \label{eq:L22}
    \int \left|C_\mu(f)(x) \right|^2 \, d\mu(x) & = \int \sum_{k \in \Z} |\cmuk(f)(x)|^2 \, d\mu(x) \\
    & = \sum_{ k \in \Z} \int |\cmuk(f)(x)|^2 \,  d\mu(x) \\
    & = \sum_{k \in \Z} \sum_{Q \in \mucubes^k}\int_Q |\cmuk(f)(x)|^2 \, d\mu(x)\\
    & = \sum_{Q \in \mucubes} \int_{\R^d} \left| \chara_Q (x) C_{\mu, J(Q)} (f)(x) \right|^2 \, d\mu(x).
\end{align}

\begin{notation}
For $Q \in \mucubes$, 
\begin{align*}
    C_Q(f)(x) := \chara_Q(x) C_{\mu, J(Q)}(f)(x).
\end{align*}
\end{notation}

Writing this out:
\begin{align*}
    C_Q(f)(x) & = \chara_Q(x) \cmuq (f)(x) \\
    & =\chara_Q(x) \int_{\ell(Q)/4}^{\ell(Q)/2} \left| \frac{1}{t^n} \int_{B(x,t)} \quotyx f(y)\, d\mu(y) \right|^2 \, \dt.
\end{align*}
Thus, for fixed $x \in Q$, $|x-y| \leq \ell(Q)/2$ and thus
\begin{align*}
    y \in N_{\ell(Q)/2} (Q),
\end{align*}
the $\frac{\ell(Q)}{2}$-neighbourhood of Q. If we set
\begin{align*}
    \neig(Q):= \left\{ P \in \mucubes \, |\, \ell(P) = \ell(Q) \mbox{  and  } \dist(P, Q) \leq \ell(Q)\right\},
\end{align*}
we see therefore that 
\begin{align} \label{eq:L21}
y \in \bigcup_{P \in \neig(Q)} P.
\end{align}
\begin{notation}
For $Q \in \mucubes$, set 
\begin{align*}
    N(Q):= \bigcup_{P \in \neig(Q)} P.
\end{align*}
\end{notation}

With \eqref{eq:L21}, we see that
\begin{align*}
    \cmuq(f)(x) = \cmuq(\chara_{N(Q)} f)(x).
\end{align*}

Now, we may decompose $\chara_{N(Q)} f$ through a martingale decomposition. That is, we may write 
\begin{align*}
\chara_{N(Q)}(x)f(x) = \sum_{R \in \neig(Q)}\left( \chara_R(x)f_R + \sum_{P \in \mucubes(R)}\Delta_P f(x) \right).
\end{align*}
Recall that
\begin{align*}
    \Delta_Pf(x) = \sum_{S \in \mucubes^1(P)}\left( f_S - f_P\right) \chara_S(x).
\end{align*}
We write
\begin{align*}
    &\sum_{R \in \neig(Q)}\left( \chara_R(x)f_R + \sum_{P \in \mucubes(R)}\Delta_P f(x) \right)\\ & = \sum_{R \in \neig(Q)}\chara_R(x) f_R + \sum_{R \in \neig(Q)}\sum_{P \in \mucubes(R)} \Delta_Pf(x)  \\
    & = \sum_{R \in \neig(Q)}\chara_R(x) f_R + \sum_{R \in \neig(Q)} \chara_R(x) f_Q - \sum_{R \in \neig(Q)} \chara_R(x) f_Q + \sum_{R \in \neig(Q)}\sum_{P \in \mucubes(R)} \Delta_Pf(x)  \\
    & = \chara_{N(Q)}(x) f_Q + \sum_{R \in \neig(Q)}\chara_R(x)(f_R - f_Q) + \sum_{R \in N(Q)} \sum_{P \in \mucubes(R)} \Delta_P f(x).
\end{align*}
Hence we  split $\cmuq$ as follows. 
\begin{align*}
   & \cmuq(f)(x) \\
   & = \chara_Q (x) \left(\int_{\ell(Q)/4}^{\ell(Q)/2} \left| t^{-n} \int_{B(x,t)} \quotyx \left(\chara_{N(Q)} f\right) (y) d\mu(y) \right|^2 \dt \right)^{1/2} \\
    & \leq \chara_Q (x) \left(\int_{\ell(Q)/4}^{\ell(Q)/2} \left| t^{-n} \int_{B(x,t)} \quotyx \left(\chara_{N(Q)} f_Q\right) (y) d\mu(y) \right|^2 \dt \right)^{1/2} \\
    & \enskip \enskip + \chara_Q (x) \left(\int_{\ell(Q)/4}^{\ell(Q)/2} \left| t^{-n} \int_{B(x,t)} \quotyx \left(\sum_{R \in \neig(Q)}\chara_R(y)(f_R-f_Q)\right) d\mu(y) \right|^2 \dt \right)^{1/2} \\
    & \enskip \enskip + \chara_Q (x) \left(\int_{\ell(Q)/4}^{\ell(Q)/2} \left| t^{-n} \int_{B(x,t)} \quotyx \left(\sum_{R \in \neig(Q)} \sum_{P \in \mucubes(R)} \Delta_Pf(y)\right)  d\mu(y) \right|^2 \dt \right)^{1/2}.
\end{align*}

Thus, also recalling \eqref{eq:L22}, we see that
\begin{align*}
    & \int \left|C_\mu(f)(x)\right|^2 \, d\mu(x) \\
    & \leq \sum_{Q \in \mucubes} \int_Q |f_Q|^2 |C_\mu(\chara_{N(Q)})(x)|^2\, d\mu(x) \\
    & \enskip \enskip + \sum_{Q \in \mucubes} \int_Q \left| \left(\sum_{R \in \neig(Q)} f_R - f_Q\right) C_\mu(\chara_R)(x) \right|^2 \, d\mu(x) \\
    & \enskip \enskip + \sum_{Q \in \mucubes} \int_Q \left| \sum_{R \in \neig(Q)} C_\mu\left( \sum_{P \in \mucubes(R)} \Delta_Pf\right)(x) \right|^2\, d\mu(y) \\
    & =: A + B + C.
\end{align*}

\subsubsection{Estimates for $A$ and $B$}
The terms $A$ and $B$ may be estimated as the first and second term in equation 4.1, page 6 of \cite{tolsatoro}, to obtain
\begin{align*}
    A+ B \lesssim \|f\|_{L^2(\mu)}^2.
\end{align*}

\subsubsection{Estimate for $C$}
We write
\begin{align*}
    & \sum_{Q \in \mucubes} \int_Q \int_{\ell(Q)/4}^{\ell(Q)/2} \left|\sum_{R \in \neig(Q)} t^{-n}\int_{B(x,t)} \quotyx \left(\sum_{P \in \mucubes(R)} \Delta_Pf\right)(y) \, d\mu(y) \right|^2 \, \dt \, d\mu(x) \\
    & \leq \sum_{Q \in \mucubes} \sum_{P \in \neig(Q)} \int_Q \int_{\ell(Q)/4}^{\ell(Q)/2}\left| \sum_{P \in \mucubes(R)} t^{-n}\int_{B(x,t)} \quotyx \Delta_P(y) \, d\mu(y) \right|^2\, \dt \mu(x) \\
    & \lesssim  \sum_{Q \in \mucubes} \sum_{R \in \neig(Q)} \int_Q \int_{\ell(Q)/4}^{\ell(Q)/2}\left| \sum_{\substack{P \in \mucubes(R)\\ P \cap \partial B(x,t) = \emptyset}} t^{-n}\int_{B(x,t)} \quotyx \Delta_P(y) \, d\mu(y) \right|^2\, \dt \mu(x)\\
    & \enskip \enskip + \sum_{Q \in \mucubes} \sum_{R \in \neig(Q)} \int_Q \int_{\ell(Q)/4}^{\ell(Q)/2}\left| \sum_{\substack{P \in \mucubes(R)\\ P \cap \partial B(x,t) \neq \emptyset}} t^{-n}\int_{B(x,t)} \quotyx \Delta_P(y) \, d\mu(y) \right|^2\, \dt \mu(x)\\
    & = : C_1 + C_2.
\end{align*}
Notice that if $P \subset B(x,t)^c$, then it is negligible for our computation. The term $C_2$ may be estimated as the third term in in equation 4.1, page 6 of \cite{tolsatoro} (and recalling Proposition \ref{proposition:ns_ur_c}).

To estimate $C_1$, we notice that, because $P \subset B(x,t)$, then 
\begin{align*}
    \int_{B(x,t)} x \Delta_Pf(y) \, d\mu(y) =0.
\end{align*}
Thus, letting $c_P$ to be the center of $P$, recalling that $\spt\left(\Delta_P f\right) \subset P$, and by Cauchy - Scwhartz, we have
\begin{align*}
    & \left| \sum_{\substack{P \in \mucubes(R)\\ P \cap \partial B(x,t) \neq \emptyset}} t^{-n}\int_{B(x,t)} \quotyx \Delta_P(y) \, d\mu(y) \right|\\
    & = \left| \sum_{\substack{P \in \mucubes(R)\\ P \cap \partial B(x,t) \neq \emptyset}} t^{-n}\int_{B(x,t)} \ps{\frac{c_P-y}{t} }\Delta_P(y) \, d\mu(y) \right|  \\
    & \lesssim \sum_{\substack{P \in \mucubes(R)\\ P \cap \partial B(x,t) = \emptyset}} \frac{1}{\mu(Q)} \frac{\ell(P)}{\ell(Q)} \|\Delta_Pf\|_{L^1(\mu)} \\
    & \lesssim \sum_{\substack{P \in \mucubes(R)\\ P \cap \partial B(x,t) = \emptyset}} \frac{1}{\mu(Q)} \frac{\ell(P)}{\ell(Q)} \ell(P)^{\frac{n}{2}}\|\Delta_Pf\|_{L^2(\mu)}
\end{align*}
Thus
\begin{align*}
    C_1 \lesssim & \sum_{Q \in \mucubes} \sum_{R \in \neig(Q)} \int_Q \int_{\ell(Q)/4}^{\ell(Q)/2} \left( \sum_{\substack{P \in \mucubes(R)\\ P \cap \partial B(x,t) = \emptyset}} \frac{1}{\mu(Q)} \frac{\ell(P)}{\ell(Q)} \ell(P)^{\frac{n}{2}}\|\Delta_Pf\|_{L^2(\mu)} \right)^2 \\
    & \lesssim  \sum_{Q \in \mucubes} \sum_{R \in \neig(Q)} \sum_{P \in \mucubes(R)} \left( \frac{\ell(P)}{\ell(Q)} \ell(P)^{\frac{n}{2}}\|\Delta_Pf\|_{L^2(\mu)}\right)^2 \, \frac{1}{\mu(Q)} \\
    & \leq \sum_{Q \in \mucubes} \sum_{R \in \neig(Q)} \left( \sum_{P \in \mucubes(R)} \frac{\ell(P)}{\ell(Q)} \|\Delta_Pf\|_{L^2(\mu)}^2 \right) \left(\sum_{P \in \mucubes(R)} \frac{\ell(P)^{n+1}}{\ell(Q)} \right) \, \frac{1}{\mu(Q)}
\end{align*}
We see that 
\begin{align*}
    \sum_{P \in \mucubes(R)} \frac{\ell(P)^{n+1}}{\ell(Q)} \lesssim \mu(Q).
\end{align*}
Finally, by Fubini, 
\begin{align*}
    & \sum_{Q \in \mucubes} \sum_{R \in \neig(Q)} \left( \sum_{P \in \mucubes(R)} \frac{\ell(P)}{\ell(Q)} \|\Delta_Pf\|_{L^2(\mu)}^2 \right) \\
    & = \sum_{P \in \mucubes} \|\Delta_Pf\|_{L^2(\mu)}^2 \sum_{\substack{Q \in \mucubes \\ Q \supset P}} \sum_{R \in \neig(Q)} \frac{\ell(P)}{\ell(Q)}.
\end{align*}
Since the cardinality of $N(Q)$ is bounded above by a universal constant, we see that 
\begin{align*}
    \sum_{\substack{Q \in \mucubes \\ Q \supset P}} \sum_{R \in \neig(Q)} \frac{\ell(P)}{\ell(Q)} \lesssim \sum_{\substack{Q \in \mucubes \\ Q \supset P}} \frac{\ell(P)}{\ell(Q)} \lesssim 1.
\end{align*}
Thus we obtain 
\begin{align*}
C_1 \lesssim \sum_{P \in \mucubes} \|\Delta_P f\|_{L^2(\mu)}^2.
\end{align*}
This together with the estimates on $A,B$ and $C_2$ proves Proposition \ref{proposition:L2}

\subsection{$L^{1,\infty}(\mu)$ boundedness of $C_\mu$}
Let $M(\R^d)$ be the set of finite Borel measures on $\R^d$ and $\nu \in M(\R^d)$. Let the operator $C$ on $M(\R^d)$ be given by
\begin{align*}
    C_\nu(x) := \left(\int_0^\infty |C_\nu(x,t)|^2 \, \dt \right)^{\frac{1}{2}}.
\end{align*}
\begin{proposition} \label{proposition:L1weak}
Let $\mu$ be an $n$-uniformly rectifiable measure on $\R^d$. Then, 
\begin{align*}
    \mu\left( \left\{ x \in \R^d \, |\, C_\nu(x) > \lambda \right\}\right) \leq \frac{\|\nu\|}{\lambda}
\end{align*}
for each $\nu \in M(\R^d)$ with compact support and $\lambda>0$.
\end{proposition}

We will use the Calder\'{o}n-Zygmund decomposition for $\nu \in M(\R^d)$. See Theorem in \cite{tolsa-book}.

\begin{theorem} \label{theorem:czmeasure}
Let $\mu$ be an $n$-AD-regular measure on $\R^d$. For every $\nu \in M(\R^d)$ with compact support and every $\lambda > 2^{d+1}\|\nu\|/\|\mu\|$ we have:
\begin{enumerate}
    \item There exists a finite or countable collection of dyadic cubes $\{D_j\}_{j \in J} \subset \dcubes$ with \begin{align} \label{eq:czmeasure1}
        \sum_{j \in J} \chara_{D_j} \lesssim 1,
    \end{align}
    and a function $f \in L^1(\mu)$ such that, for each $j \in J$, 
    \begin{align}
        & \mu(2D_j) < \frac{2^{d+1}}{\lambda} |\nu|(D_j). \label{eq:czmeasure2}\\
        & \frac{2^{d+1}}{\lambda} \nu(\eta D_j) \leq \mu(2 \eta D_j) \mbox{  for every  } \eta> 2. \label{eq:czmeasure3}\\
        & \nu = f \mu \enskip \mbox{ in } \enskip \R^d \setminus \left( \bigcup_{j \in J} D_j \right) \enskip \mbox{ with } \enskip |f| \leq \lambda \,\, \, \,  \mu-a.e.\label{eq:czmeasure4}
    \end{align}
    \item For each $j \in J$, set 
    \begin{align}
        & R_j := 6D_j \label{eq:czmeasure5}\\
        & w_j := \chara_{D_j} \left(\sum_{i \in J} \chara_{D_i}\right)^{-1}.\label{eq:czmeasure6}
    \end{align}
    There exsits a family of functions $\{b_j\}_{j \in J}$ with 
    \begin{align}
        \spt(b_j) \subset R_j,\label{eq:czmeasure7}
    \end{align}
    each one with constant sign, such that
    \begin{align}
    & \int b_j d\mu = \int w_j d\nu \label{eq:czmeasure8} \\
    & \|b_j \|_{L^\infty(\mu)} \mu(R_j) \leq c|\nu|(D_j) \label{eq:czmeasure9} \\
    & \sum_{j \in J} |b_j| \leq c \lambda. \label{eq:czmeasure10}
    \end{align}
\end{enumerate}
\end{theorem}

\begin{proof}[Proof of Proposition \ref{proposition:L1weak}]
We keep the notation as in Theorem \ref{theorem:czmeasure}. Let $\nu \in M(\R^d)$ and fix $\lambda > 2^{d+1}\frac{\|\nu\|}{\|\mu\|}$. We write
\begin{align*}
    d\nu & = \chara_{\R^d \setminus \cup_{j \in J} D_j} d\nu + \chara_{\cup_{j \in J}}d\nu 
     = \chara_{\R^d \setminus \cup_{j \in J} D_j}\, f\, d\mu + \sum_{j \in J} w_j \, d\nu\\  
     &= \chara_{\R^d \setminus \cup_{j \in J} D_j}\, f\, d\mu + \sum_{j \in J} b_j \, d\mu + \sum_{j \in J} \left(w_j d\nu - b_j d\mu\right).
\end{align*}
Thus if we set 
\begin{align*}
   & F:=\bigcup_{j \in J} D_j , \, \, 
     g:= f \chara_{F^c} + \sum_{j \in J} b_j,\, \, 
     d\beta := \sum_{j \in J} (w_j d\nu - b_j d\mu),
\end{align*}
we may let 
\begin{align} \label{eq:L1weak1}
    d\nu = g \,d\mu + d\beta.
\end{align}

Let $2F = \cup_{j \in J} 2D_j$. Notice first that, because \eqref{eq:czmeasure2}, 
\begin{align*}
    \mu\ps{\left\{ x \in 2F  \, |\, C\nu(x) > \lambda\right\}} & \leq \mu( 2F) 
     \leq \sum_{j \in J} \mu(2D_j ) 
    \leq \frac{2^{d+1}}{\lambda} |\nu|(D_j)
     \lesssim_d \frac{\|\nu\|}{\lambda}.
\end{align*}
Hence we may work on $(2F)^c$ only.
We split $C_\nu$ as suggested by \eqref{eq:L1weak1}: 
\begin{align*}
C_\nu(x) & = \left( \int_0^\infty \left| t^{-n} \int_{B(x,t)} \quotxy \left( g(y)d\mu(y) + d \beta(y)\right) \right|^2 \right)^{1/2}\\
& \leq \left( \int_0^\infty \left| t^{-n} \int_{B(x,t)} \quotxy  g(y)\, d\mu(y) \right|^2 \right)^{1/2}\\
& \enskip \enskip + \left( \int_0^\infty \left| t^{-n} \int_{B(x,t)} \quotxy  d \beta(y) \right|^2  \right)^{1/2}\\
& = C_\mu(g)(x) + C_\beta(x).
\end{align*}

Then
\begin{align*}
    & \mu\left(\left\{ x \in (2F)^c \, |\, C_\nu(x) >\lambda\right\} \right)  
     \leq \mu(\left( \left\{ x \in (2F)^c \, |\, C_\mu(g)(x) > \frac{\lambda}{2} \right\} \right) + \mu\left(\left\{ x \in  (2F)^c \, |\, C_\beta(x) > \frac{\lambda}{2} \right\} \right)\\
    & = : A + B.
\end{align*}

\subsubsection{Estimates for $A$}
This is easily done by noticing that $g \in L^2(\mu)$, since $g \in L^\infty(\mu)$ and $\mu(\R^d)< \infty$. Thus, in particular, 
\begin{align*}
\mu\ps{\left\{ x \in (2F)^c \, |\, C_\mu(g)(x) > \lambda/2 \right\}} \leq \frac{C}{\lambda^2} \int |g|^2\, d\mu \lesssim \frac{1}{\lambda} \int|g|\, d\mu,
\end{align*}
since, by \eqref{eq:czmeasure4} and \eqref{eq:czmeasure10}, 
\begin{align*}
    \|g\|_{L^\infty(\mu, )} & \leq \|f\chara_{F^c}\|_{L^\infty(\mu)} + \left\|\sum_{j \in J}b_j\right\|_{L^\infty(\mu)} 
     \leq c\lambda.
\end{align*}
Moreover, using some of the properties listed in Theorem \ref{theorem:czmeasure}, 
\begin{align*}
    \int |g| \, d\mu & \leq \int |f\chara_{F^c}| \, d\mu + \int \left| \sum_{j \in J} b_j \, d\mu \right|
     = |\nu|(F^c) + \sum_{j \in J} \int |b_j| \, d\mu 
     \leq \|\nu\| + \sum_{j \in J} \|b_j\|_\infty \mu(R_j) \\
    & \leq \|\nu\| + \sum_{j \in J} c |\nu|(D_j) 
     \lesssim \|\nu\|.
\end{align*}
Thus
\begin{align*}
    \mu\ps{\left\{ x \in (2F)^c \, |\, C_\mu(g)(x) > \lambda/2 \right\}} \lesssim \frac{\|\nu\|}{\lambda}.
\end{align*}

\subsubsection{Estimates for $B$}
Set $E:= (2F)^c$. By Chebyshev's inequality, 
\begin{align*}
    \mu\ps{\ck{ x \in E \, |\, C_\beta(x) > \lambda/2}} & \leq \frac{2}{\lambda} \int_E C_\beta(x) \, d\mu(x) \\
    & \leq \frac{2}{\lambda} \sum_{j \in J} \int_{\R^d \setminus 2 R_j} C_{\beta_j} (x) \, d\mu(x) 
     + \frac{2}{\lambda} \sum_{j \in J} \int_{2R_j \setminus 2D_j} C_{\beta_j}(x) \, d\mu(x) 
     =: B_1 + B_2.
\end{align*}

\textit{Estimates for $B_1$.}
We write
\begin{align*}
    \int_{\R^d \setminus R_j} C_{\beta_j}(x) & =  \int_{\R^d \setminus R_j} \left(\int_0^\infty \av{t^{-n-1} \int_{B(x,t)} (x-y) \, d \beta_j(y)}^2 \, \dt\right)^{1/2} \, d\mu(x)\\
    & =  \int_{\R^d \setminus R_j} \ps{\int_{\ck{t: R_j \subset B(x,t)^c}} \av{t^{-n-1} \int_{B(x,t)} (x-y) \, d \beta_j(y)}^2 \, \dt }^{1/2} \, d\mu(x)
    \\ & +  \int_{\R^d \setminus R_j} \ps{\int_{\ck{t: R_j \subset B(x,t)}} \av{t^{-n-1} \int_{B(x,t)} (x-y) \, d \beta_j(y)}^2 \, \dt }^{1/2} \, d\mu(x)\\
    & +  \int_{\R^d \setminus R_j} \ps{\int_{\ck{t: R_j \cap B(x,t)}} \av{t^{-n-1} \int_{B(x,t)} (x-y) \, d \beta_j(y)}^2 \, \dt }^{1/2} \, d\mu(x)\\
    & := B_{1,1} + B_{1,2} + B_{1,3} 
\end{align*}

Clearly $B_{1,1} =0$; moreover, since $x\in \R^d \setminus 2R_j$, if $t\leq \dist(x, R_j)$, then $t \in \ck{t: R_j \subset B(x,t)^c}$. 

We estimate $B_{1,2}$. First, notice that $\beta_j (R_j) = 0$. Thus, if $t \in \ck{t: R_j \cap B(x,t)}$ and letting $c_j$ denote the center of $R_j$, we get
\begin{align*}
    \av{ \int_{B(x,t)}t^{-n-1} (x-y) \, d \beta_j (y) } & = \av{ \int_{B(x,t)} t^{-n-1} (c_j - y) \, d\beta_j(y)} 
    \lesssim \ps{\int_{B(x,t)} t^{-n}\frac{\ell(R_j)}{|x-c_j|} d|\beta_j|(y)} \\ 
    & = t^{-n}\frac{\ell(R_j)|\beta_j|(R_j)}{|x-c_j|}.
\end{align*}
Since, if $t \in \ck{t: R_j \cap B(x,t)}$, then $t>\dist(x, R_j)$, then
\begin{align*}
C_{\beta_j}(x)^2 & \lesssim \int_{\dist(x, R_j)}^\infty \left(t^{-n}\frac{\ell(R_j)|\beta_j|(R_j)}{|x-c_j|}\right)^2 \, \dt 
 = \frac{\ell(R_j)^2|\beta_j|(R_j)^2}{|x-c_j|^2} \int_{\dist(x, R_j)}^\infty t^{-n-1} \, dt \\
& \lesssim \frac{\ell(R_j)^2|\beta_j|(R_j)^2}{|x-c_j|^2} \frac{1}{|x-c_j|^{2n}}.
\end{align*}
Recalling that $\mu$ is $n$-uniformly rectifiable and thus $n$-AD-regular, 
\begin{align*}
\int_{\R^d \setminus 2R_j} \frac{1}{|x-c_j|^{n+1}} \, d\mu(x) & = \sum_{k=0}^\infty \int_{2^k (2\ell(R_j)) \leq |c_j - x| \leq 2^{k+1}(2\ell(R_j))} \frac{1}{|x-c_j|^{n+1}} \, d\mu(x) \\
& \leq \sum_{k=0}^\infty \frac{\mu(B(c_j, 2^{k+1}2\ell(R_j)))}{2^{(n+1)k} (2\ell(R_j))^{n+1} }\\
&\lesssim_{c_0} \sum_{k=0}^\infty \frac{2^{(k+1)n}(2\ell(R_j))^n}{2^{(n+1)k} (2\ell(R_j))^{n+1} }\\
&\lesssim \frac{1}{\ell(R_j)}.
\end{align*}
All in all, we then see that
\begin{align*}
    \int_{\R^d \setminus 2R_j} C_{\beta_j}(x) \, d\mu(x) & \lesssim \ell(R_j) |\beta_j|(R_j)\int_{\R^d \setminus 2R_j}\frac{1}{|x-c_j|^{n+1}} \, d\mu(x)
    \lesssim |\beta_j|(R_j).
\end{align*}

One can easily estimate $B_{1,3}$ by arguing as Tolsa and Toro in \cite{tolsatoro}; see the treatment of the first term on the right hand side of equation 5.10 at page 10 there.

\bigskip 

\textit{Estimates for $B_2$.}
Applying Cauchy-Schwartz, We write
\begin{align*}
    \int_{2R_j \setminus 2R_j} C_{\beta_j}(x) \, d\mu(x) & \leq \mu(R_j)^{n/2} \left( \int_{2R_j \setminus 2R_j} |C_{\beta_j}(x)|^2 \, d\mu(x)\right)^{1/2} \\
    & \leq \mu(R_j)^{n/2} \left( \int_{2R_j \setminus 2R_j} |C_{\mu}(b_j)(x)|^2 \, d\mu(x)\right)^{1/2} \\
    & \enskip \enskip + \mu(R_j)^{n/2} \left( \int_{2R_j \setminus 2R_j} |C_{\nu}(w_j)(x)|^2 \, d\mu(x)\right)^{1/2}.
\end{align*}

Since  $b_j \in L^\infty(\mu)$ with compact support, then $b_j \in L^2(\mu)$. Thus, by Proposition \ref{proposition:L2}, we see that
\begin{align*}
    \mu(R_j)^{n/2} \left( \int_{2R_j \setminus 2R_j} |C_{\mu}(b_j)(x)|^2 \, d\mu(x)\right)^{1/2}  & \lesssim \mu(R_j)^{n/2} \|b_j\|_{L^2(\mu)} \\
    & \leq \mu(R_j)^{n/2} \mu(R_j)^{n/2} \|b_j\|_\infty\\
    & \leq |\nu|(D_j).
\end{align*}
On the other hand, recalling that $\spt(w_j) \subset D_j$ and $\|w_j\|_\infty \leq 1$,  
\begin{align*}
C_\nu(w_j)(x)^2 & = \int_0^\infty \av{t^{-n} \int_{B(x,t)} (x-y) \, w_j(y) \, d\nu(y) }^2 \, \dt
 \leq \int_{\ell(R_j)}^\infty \av{t^{-n} \int_{B(x,t)} (x-y) \, w_j(y) \, d\nu(y) }^2 \, \dt\\
& \leq \int_{\ell(R_j)}^\infty \frac{\ell(R_j)^2 |\nu|(D_j)^2}{t^{2n+2}}\, d\nu(y) \, \dt 
 \leq \frac{|\nu|(D_j)^2}{\ell(R_j)^{2n}}.
\end{align*}
Hence, we see that
\begin{align*}
    \mu(R_j)^{n/2} \left( \int_{2R_j \setminus 2R_j} |C_{\nu}(w_j)(x)|^2 \, d\mu(x)\right)^{1/2} & \lesssim \mu(R_j)^{n/2} \mu(R_j)^{n/2} \frac{|\nu|(D_j)}{\ell(R_j)^n} 
    \lesssim |\nu|(D_j).
\end{align*}

All in all, recalling \eqref{eq:czmeasure1}, we see that
\begin{align*}
B_1 + B_2 \lesssim \frac{1}{\lambda} \sum_{j \in J} |\nu|(D_j) \lesssim \frac{\|\nu\|}{\lambda}.
\end{align*}
\end{proof}

As in the proof of Theorem 5.1 in \cite{tolsatoro}, we have the following corollary.
\begin{corollary}
Proposition \ref{proposition:L1weak} holds for any $\nu \in M(\R^d)$. 
\end{corollary}

\begin{proof}
Fix $N_1$ be an arbitrary positive number. Pick $N_2 \geq 2N_1$ and such that $|N_1-N_2|> \frac{2\|\nu\|}{\lambda}$. We see that, if $x \in B(0, N_1)$, 
\begin{align*}
    C_\nu(\chara_{\R^d\setminus B(0, N_2)}) (x)^2& = \int_0^\infty \av{t^{-n}\int_{B(x,t)} \quotxy \, \chara_{\R^d \setminus B(0,N_2)} \, d\nu(y)} \, \dt \\
    & \leq \int_{|N_1-N_2|}^\infty \av{t^{-n}\int_{B(x,t)} \quotxy \, \chara_{\R^d \setminus B(0,N_2)} \, d\nu(y)} \, \dt \\
    & \lesssim |\nu|(\R^d\setminus B(0, N_2))\int_{|N_1-N_2|}^\infty \frac{1}{t^{2n+1}} \, dt \\
    & = \frac{|\nu|(\R^d\setminus B(0, N_2))}{|N_1-N_2|^{2n}}.
\end{align*}
Thus, by our choice of $N_2$, we see that
\begin{align*}
    C_\nu(\chara_{\R^d\setminus B(0, N_2)}) (x) \leq C \frac{\lambda}{2}.
\end{align*}
Notice that, here, the constant $C$ does \textit{not} depend on $N_1$. 

Set $\nu_2:= \nu|_{B(0, N_2)}$. We see that
\begin{align*}
    \mu\ps{\ck{x \in B(0, N_1) \, | C_\nu(x) > \lambda }} & \leq \mu\ps{\ck{ x \in B(0,N_1) \, |\, C_{\nu_2}(x) > \lambda/2 }} \\
    & \lesssim \frac{\|\nu_2\|}{\lambda} 
     \leq \frac{\|\nu\|}{\lambda}.
\end{align*}
Since these estimates do not depend on $N_1$, we may let $N_1 \to \infty$.
\end{proof}

\section{Proof of Theorem \ref{t:UR-Omega}}
What is left to do is to prove the remaining direction in Theorem \ref{t:UR-Omega}. That is, we need to prove the following. 
\begin{proposition}\label{p:C-UR-Omega}
Let $\mu$ be an Ahlfors 1-regular measure in $\C$. Let $\Omega$ be as in the statement of Theorem \ref{t:UR-Omega}. If $|C^1_{\Omega, \mu, \phi}(x,t)|^2 \, \dt \, d\mu(x)$ is a Carleson measure on $\spt(\mu) \times (0, \infty)$, then $\mu$ is uniformly $1$-rectifiable. 
\end{proposition}
Recall that 
\begin{align*}
    C_{\Omega, \mu, \phi}^1(x,t) := t^{-1} \int \frac{|x-y|\Omega((x-y)/|x-y|)}{t} e^{-\av{\frac{x-y}{t}}^{2N}} \, d\mu(y). 
\end{align*}
The proof of Proposition \ref{p:C-UR-Omega} is a standard compactness argument. For example, one can follow \cite{tolsa2014} almost verbatim. There is one important difference, however. In the limit, one ends up with a measure  which satisfies 
$C_{\Omega, \nu, \phi}^1(x,t)=0$ for all $t>0$ and $x \in \spt(\nu)$. It then follows from an argument similar to the proof of Lemma 3.6 in \cite{tolsa2014}, that $\nu$ is an $\Omega$-symmetric measure, that is, a measure satisfying
$C_{\Omega, \nu}(x,t)=0$ for all $t>0$ and $x \in \spt(\nu)$. Then one needs to appeal to Theorem 1.4 in \cite{villa19} to conclude that such a measure is flat. With this, one can close the compactness argument and continue as in \cite{tolsa2014}.  

The remaining direction of Theorem \ref{t:UR-Omega} follows from Proposition \ref{p:C-UR-Omega} since 
\begin{align*}
    \int_{B}\int_0^{r(B)} |C_{\Omega, \mu, \phi}^1(x,t)|^2 \, \dt \, d\mu(x) \lesssim \int_B \int_0^{r(B)} |C_{\Omega, \mu}^1(x,t) |^2 \, \dt \, d\mu(x).
\end{align*}
This is a standard fact. A proof can be found for example in \cite{tolsa2014}, Corollary 3.12.

\end{document}